\documentclass[11pt,reqno]{amsart}
\usepackage[T1]{fontenc}
\usepackage{graphicx}
\usepackage{cite}
\usepackage{empheq}
\usepackage{hyphenat}

\usepackage{color}
\definecolor{MyLinkColor}{rgb}{0,0,0.4}
\numberwithin{equation}{section}

\newcommand{\re}{\mathop{\rm Re}\nolimits}

\newcommand{\PV}{\mathop{\rm PV}\nolimits}

\newcommand{\e}{\varepsilon}

\newcommand{\p}{\partial}
\newcommand{\wt}{\widetilde}
\newcommand{\ov}{\overline}

\newcommand{\oo}{\ov\omega}

\newcommand{\bA}{\mathbb{A}}

\newcommand{\bB}{\mathbb{B}}

\newcommand{\kH}{\mathcal{H}}

\newcommand{\kL}{\mathcal{L}}

\newcommand{\R}{\mathbb{R}}

\newcommand{\N}{\mathbb{N}}

\DeclareMathOperator{\supp}{supp}
 
\newtheorem{thm}{Theorem}[section]
\newtheorem{prop}[thm]{Proposition}
\newtheorem{lemma}[thm]{Lemma}
\newtheorem{cor}[thm]{Corollary}

\numberwithin{equation}{section}

\title[The Muskat problem with surface tension]{A new reformulation of the Muskat problem with surface tension}
\author{Anca--Voichita Matioc}
\address{Fakult\"at f\"ur Mathematik, Universit\"at Regensburg,   93040 Regensburg, Deutschland.}
\email{anca.matioc@ur.de}
\email{bogdan.matioc@ur.de}

\author{Bogdan--Vasile Matioc}


\usepackage[colorlinks=true,linkcolor=MyLinkColor,citecolor=MyLinkColor]{hyperref} 

\begin{document}
  
\begin{abstract} 
Two formulas that connect the derivatives of the double layer potential and of a related singular integral operator evaluated at some density $\vartheta$ to the
$L_2$-adjoints of these operators evaluated at the density $\vartheta'$ are used to recast 
  the Muskat problem with surface tension and general viscosities as a system of equations with nonlinearities expressed in terms 
of the $L_2$-adjoints of these operators.  
An advantage of this formulation is that the nonlinearities appear now as a derivative.
This aspect and abstract quasilinear parabolic theory are then exploited to establish a local well-posedness result in all subcritical Sobolev spaces $W^s_p(\R)$ with $p\in(1,\infty)$ and $s\in (1+1/p,2)$.
\end{abstract}

\maketitle

\section{Introduction}\label {Sec:1}

In this paper we consider the two-dimensional Muskat problem describing the dynamics in an unbounded two fluids system which moves with constant speed $V$ 
in a  horizontal/vertical Hele-Shaw cell or in a porous medium. 
The fluids are assumed to  fill the entire plane and
the free interface between the fluids is parameterized as the graph 
 \[
 \{(x,f(t,x)+tV)\,:\, x\in\R\}\qquad\text{ for $t\geq 0$.}
 \]
We   take into account both gravity and surface tension effects. Let therefore $\kappa(f(t))$ denote the curvature of the free interface and let~$\sigma>0$ be the surface tension coefficient.
The subscript~$-/+$ is used to denote the fluid located below/above the interface,  $g\geq0 $ is the Earth's gravity, and~${k}$ is the permeability of the homogeneous porous medium.
Moreover, the positive constants~$\mu_\pm$ and~$\rho_\pm$ are the viscosity and the density  of the fluids.
Introducing a further unknown~$\oo:=\oo(t,x)$, with~$2(1+(\p_xf)^2)^{-1/2}\oo$ measuring the jump of the velocity field in tangential direction at the interface,
the Muskat problem can be expressed in a compact form as the following coupled system
\begin{equation}\label{P:1}
\left.
\begin{array}{rlll}
\cfrac{df}{dt}(t)\!\!&=&\!\! \bB(f(t))[\oo(t)],\quad \text{$t>0$},\\[2ex]
(1-a_\mu\bA(f(t)))[\oo(t)]\!\!&=&\!\!b_\mu \big(\sigma\kappa(f(t))-\Theta f(t)\big)',\quad \text{$t>0$},\\[2ex]
f(0)\!\!&=&\!\! f_0,
\end{array}
\right\}
\end{equation} 
 cf., e.g., \cite{MBV18, AM22, Ngu20, A14}. 
 The constants in \eqref{P:1} are given by the relations
\[
b_{\mu}:=\frac{k}{\mu_-+\mu_+}>0,\qquad  a_\mu:=\frac{\mu_--\mu_+}{\mu_-+\mu_+}\in(-1,1),\qquad  \Theta:=g(\rho_--\rho_+)  +\frac{\mu_--\mu_+}{k}V\in\R.
\]
 Throughout the paper $(\,\cdot\,)'$ denotes differentiation with respect to the spatial variable  $x.$ 
 Furthermore, given a Lipschitz continuous map~$f:\R\to\R$, the singular integral operators $\bA(f)$ and~$\bB(f)$ in~\eqref{P:1} are given by
 \begin{equation}\label{OpAB}
 \begin{aligned}
 \bA(f)[\oo](x)&:=\frac{1}{\pi}\PV\int_\R\frac{f'(x)-(\delta_{[x,y]}f)/y}{1+\big[(\delta_{[x,y]}f)/y\big]^2}\frac{\oo(x-y)}{y}\, dy,\\[1ex] 
 \bB(f)[\oo](x)&:=\frac{1}{\pi}\PV\int_\R\frac{1+f'(x)(\delta_{[x,y]}f)/y}{1+\big[(\delta_{[x,s]}f)/y\big]^2}\frac{\oo(x-y)}{y}\, dy
 \end{aligned}
 \end{equation}
 for $\oo\in L_2(\R)$. 
We use $\PV$ to denote the principal value  and the shorthand notation
 \[
 \delta_{[x,y]}f:=f(x)-f(x-y),\qquad x,\, y\in\R.
 \]
These operators are bounded, that is $\bA(f),\, \bB(f)\in\kL(L_p(\R))$ for all $p\in(1,\infty),$ see Lemma \ref{L:MP0}~(i) below. 
Let $\bA(f)^*,\, \bB(f)^*\in\kL(L_2(\R))$ denote their $L_2$-adjoint operators. 
We point out that the operator~$\bA(f)^*$ is the double layer potential for the Laplace operator corresponding to the unbounded hypersurface~${\{y=f(x)\}\subset\R^2.}$

A key point in our analysis are the following identities
\begin{equation}\label{comder}
(\bA(f)^*[\vartheta])'=-\bA(f)[\vartheta'] \qquad\text{and}\qquad (\bB(f)^*[\vartheta])'=-\bB(f)[\vartheta'],
\end{equation} 
which are satisfied provided that $f\in W^{2}_p(\R)$  and $\vartheta\in W^1_p(\R),$  see Proposition~\ref{Prop:1} below.
Using~\eqref{comder}, we show  in Section~\ref{Sec:3} that \eqref{P:1} can be formulated as an evolution problem for the unknown~$(f,\vartheta)$ which reads as 
\begin{equation}\label{P:2}
\left.
\begin{array}{rlll}
\cfrac{df}{dt}(t)\!\!&=&\!\! -(\bB(f(t))^*[\vartheta(t)])',\quad \text{$t>0$},\\[2ex]
(1+a_\mu\bA(f(t))^*)[\vartheta(t)]\!\!&=&\!\!b_\mu \big(\sigma\kappa(f(t))-\Theta f(t)\big),\quad \text{$t>0$},\\[2ex]
f(0)\!\!&=&\!\! f_0.
\end{array}
\right\}
\end{equation}
 The connection between $\oo(t)$ and $\vartheta(t)$ is through the relation
\[
\oo(t)=(\vartheta(t))', \qquad t>0.
\]
We shall take advantage of the new formulation~\eqref{P:2} to establish the well-posedness of the Muskat problem with surface tension in all subcritical
  Sobolev spaces $W^{s}_p(\R),$  $ s\in(1+1/p,2)$ and~${p\in(1,\infty).}$
Compared to  \eqref{P:1}, the formulation  \eqref{P:2} has  several advantages. 
On the one hand, it enables us   to  consider an  equation related to  $\eqref{P:2}_2$  in the Sobolev space $W^{\ov s-1}_p(\R),$  $\ov s\in(1+1/p,s)$, see equation~\eqref{theta},  
whereas in the context of~$\eqref{P:1}_2$  one more derivative appears on the right of this equation (and the right hand side is in the latter case a distribution). 
On the other hand,  the equation~$\eqref{P:2}_1$ will be considered in a  Sobolev space with a negative index, more precisely  in~${W^{\ov s-2}_p(\R),}$ 
but its  right side is the derivative of function that lies in $W^{\ov s-1}_p(\R)$, and this  is very useful when establishing estimates.

Exploiting also the quasilinear character of the curvature operator $\kappa(f),$ the core of our analysis is to show  that \eqref{P:2} can be recast as a quasilinear parabolic evolution problem for $f$.
Based on these properties an on the abstract quasilinear  parabolic theory presented in \cite{Am93} (see also \cite{MW20}), 
 we  prove the following local well-posedness result.

\pagebreak

\begin{thm}\label{MT1}
Let  $p\in(1,\infty)$, $1+1/p<\ov s<s<2$, and chose (an arbitrary small) $\zeta\in(0,(s-\ov s)/3]$.
Then,  given $f_0\in W^s_p(\R)$, there exists a unique maximal solution $f:=f(\,\cdot\,; f_0)$ to \eqref{P:2} such that
\begin{equation*}  
 f\in {\rm C}([0,T^+),W^s_p(\mathbb{R}))\cap {\rm C}((0,T^+), W^{{\ov s}+1}_p(\mathbb{R}))\cap {\rm C}^1((0,T^+), W^{{\ov s}-2}_p(\mathbb{R})) 
  \end{equation*}
   and  
  \[
f\in   {\rm C}^{\zeta}([0,T^+), W^{\ov s}_p(\mathbb{R}),
  \]
  where $T^+=T^+(f_0)\in(0,\infty]$ is the maximal existence time.
Moreover,  $[(t,f_0)\mapsto f(t;f_0)]$ defines a semiflow on $W^s_p(\R)$ which is smooth in the open set
  \[
\{(t,f_0)\,:\, f_0\in W^s_p(\R),\, 0<t<T^+(f_0)\}\subset \R\times W^s_p(\R)  
  \]
  and
    \begin{equation}\label{eq:fg}
 f\in {\rm C}^\infty((0,T^+)\times\R,\R)\cap {\rm C}^\infty ((0, T^+),  W^k_p(\R))\quad \text{for all $k\in\N$}.
  \end{equation}
 \end{thm}

There is a vast, partially quite recent, mathematical literature on the  Muskat problem with surface tension and its one-phase version, the Hele-Shaw problem with surface tension.
The studies investigate different important aspects of the models such as the well-posedness, cf. e.g. \cite{A14, To17, ES97, ES97a, Ch93,  EGW13, EM11a,  EMM12a,  EMW18, PS16x, PS16, HTY97, BV11},
the existence of global (weak or strong) solutions \cite{GGPS19x, CP93, GGS20, JKM21},
the  stability properties of the stationary solutions \cite{EEM09c, MBV20, EM11a,  EMM12a,  PS16x, PS16, FT03, MW20, EGW13, GGPS19x},
the zero surface tension limit of the problems \cite{A14, FN21}, and the singular limit when the thickness of the  layers (or a  nondimensional parameter) vanishes \cite{EMM12, GS19, MP12, BG22x}.

The formulation \eqref{P:1} of the unconfined Muskat problem with surface tension considered herein is derived from the classical formulation \cite{Mu34} by using
  potential theory (see  \cite{DR84} for a first result in this direction).
 The advantage of this formulation compared to the classical one is that now the equations of motion can be studied under quite general assumptions on the function $f$ and this leads to quite optimal results.
Indeed, in the references \cite{MBV19, MBV18, MBV20} the well-posedness  of the problem is established for  $H^{2+\e}$-initial data, with~$\e\in(0,1)$ arbitrarily small, and these results were improved in
 the very recent papers \cite{Ngu20, FN21} where the initial data are taken from~$H^{1+\frac{d}{2}+\e}(\R^d)$, with $\e>0$ arbitrarily small and $d\geq1$.
 It is important to point out that $W_p^{1+\frac{d}{p}}(\R^d)$ is a critical space for \eqref{P:1}, see \cite{MM21, Ngu20}. 
 For the restrictive range~${p\in (1,2)}$ the well-posedness of \eqref{P:1} in $W_p^{1+\frac{1}{p}+\e}(\R)$, again with $\e>0$ arbitrarily small, was establish recently in \cite{MM21} 
 in the particular case of fluids with equal viscosities (by using a different approach than in this paper).
  Finally, the   stability properties of equilibria to the periodic version of \eqref{P:1} have been studied in \cite{MW20, MBV20}. 
  
Our main result in Theorem~\ref{MT1} extends the well-posedness theory to all $L_p$-based subcritical Sobolev spaces $W_p^{s}(\R)$, $s\in(1+1/p,2)$ and $p\in(1,\infty)$ 
in the general case~${\rho_--\rho_+,\, \mu_--\mu_+\in\R}$.
An important aspect in the analysis is the invertibility of the operator $\lambda-\bA(f)^*$, $\lambda\in\R\setminus(-1,1)$ and~${f\in W^s_p(\R)}$, in~${L_p(\R)}$.
In the case of a bounded Lipschitz domain, when $f$ is merely Lipschitz continuous,
 this property  is a deep result of harmonic analysis, see \cite{Ve84}. 
In the unbounded setting considered herein we establish this property directly by a using   different strategy than in the bounded case \cite{Ve84}. 
We also mention that there are not so many references that consider the Muskat problem in a classical $L_p$-setting with $p\neq2$ and, apart from the references  \cite{AM22, MM21}, we only  add the paper \cite{CGSV17}
where the  particular case $\sigma=0$ and $\mu_-=\mu_+$ is considered.

\subsection*{Notation} Given $n\in\N$ and Banach spaces $E,\,E_1,\ldots,E_n,\,F$, $n\in\N$, we write $\kL^n\big(\prod_{i=1}^n E_i,F\big)$ to denote the 
Banach space of bounded $n$-linear maps from $\prod_{i=1}^n E_i$ to $F$, and 
 $\kL^n_{\rm sym}(E,F)$ stands for the space of $n$-linear, bounded, and symmetric maps $A: E^n\to F$.
 Moreover, ${\rm C}^{-1}(E,F)$ (resp.~${{\rm C}^{\infty}(E,F)}$) is the space of locally Lipschitz continuous (resp. smooth) mappings from~$E$ to~$F$ and ${\rm Isom}(E,F)$ is the open subset of the Banach space $\kL(E,F)$
 of bounded operators which consists of isomorphisms.
Given an interval $I\subset \R$  and~${\alpha\in(0,1)}$, we denote by ${\rm C}^{n}(I,E)$ the set of $n$-times continuously differentiable functions and 
${\rm C}^{n+\alpha}(I,E)$ is the subset of  ${\rm C}^{n}(I,E)$ that contains only functions with locally $\alpha$-H\"older continuous $n$th derivative.
 Furthermore,  ${\rm BUC}^{n}(I,E)$  denotes the Banach space of functions with bounded and uniformly continuous derivatives up to order~$n$ and ${{\rm BUC}^{n+\alpha}(I,E)}$ is the subspace of  ${\rm BUC}^{n}(I,E)$
 which consists only of functions with  uniformly $\alpha$-H\"older continuous $n$th derivative. 
Further, let  ${\rm BUC}^{\infty}(I,E)=\cap_{n\in\N} {\rm BUC}^{n}(I,E).$

 Moreover, following~\cite{Am95}, given Banach spaces $E_i$, $i=0,\, 1$, with dense embedding $E_1\hookrightarrow E_0,$ we~set
\begin{equation}\label{neggen}
\kH(E_1,E_0)=\{A\in\kL(E_1,E_0)\,:\, \text{$-A$ generates an analytic semigroup in $\kL(E_0)$}\}.
\end{equation}

 Given $k\in\N$ and $p\in(1,\infty),$ we further let $W^k_p(\R)$ denote the standard $L_p$-based Sobolev space with the usual  $\|\cdot\|_{W^k_p}$-norm.
Moreover, if $0<s\not\in\mathbb{N}$ with $s=[s]+\{s\}$, where~${[s]\in\mathbb{N}}$ and  $\{s\}\in(0,1)$, the  Sobolev space~$W^s_p(\mathbb{R})$ is a Banach space with the norm
$\|f\|_{W^s_p}:=\big(\|f\|_{W^{[s]}_p}^p+[f]_{W^{s}_p}^p\big)^{1/p},$
where
\begin{equation}\label{normwsp}
[f]_{W^{s}_p}^p:=\int_{\mathbb{R}^2}\frac{|f^{([s])}(x)-f^{([s])}(y)|^p}{|x-y|^{1+\{s\}p}}\, d(x,y)=\int_{\mathbb{R}}\frac{\|f^{([s])} -\tau_\xi f^{([s])}\|_p^p}{|\xi|^{1+\{s\}p}}\, d\xi.
\end{equation}
Here and throughout the text  $\{\tau_\xi\}_{\xi\in\mathbb{R}}$ denotes the $C_0$-group of  left shifts, that is~${\tau_\xi f(x):=f(x+\xi),}$ and  $\|\cdot\|_p:=\|\cdot\|_{L_p(\mathbb{R})}$.
Finally, for  $s<0$, $W^s_p(\R)$ is defined as the dual  of $W^{-s}_{p'}(\R)$.

Some of our arguments use the  well-known interpolation property  
 \begin{align}\label{IP}
W^{(1-\eta)s_1+\eta s_2}_p(\mathbb{R})=(W^{s_1}_p(\mathbb{R}), W^{s_2}_p(\mathbb{R}))_{\eta,p}, \quad -\infty< s_1<s_2<\infty, \, (1-\eta)s_1+\eta s_2\not\in\mathbb{Z},
 \end{align}
where $(\cdot,\cdot)_{\vartheta,p}$, $\vartheta\in(0,1), $ is the real interpolation functor of exponent $\eta$ and parameter~$p$, cf., e.g.,~\cite{Tr78}.
We also recall the following estimate
 \begin{align}\label{MES}
 \|gh\|_{W^{r}_p}\leq  2(\|g\|_\infty\|h\|_{W^{r}_p}+\|h\|_\infty\|g\|_{W^{r}_p}),\quad\text{$g,\, h\in W^{r}_p(\R)$,}
\end{align}
  cf., e.g., \cite[Equation (2.1)]{AM22}, which holds for $r\in(1/p,1) $ and $p\in(1,\infty)$.

\subsection*{Outline} In Section~\ref{Sec:2} we introduce a family of singular integral operators that is needed in the analysis and we establish the relations \eqref{comder}.
 Section~\ref{Sec:3}  provides several  invertibility  results for~${\lambda-\bA(f)^*}$ and the equivalence of the formulations~\eqref{P:1} and~\eqref{P:2}.
 Finally, in Section~\ref{Sec:4} we formulate~\eqref{P:2} as a quasilinear parabolic evolution equation for~$f$  and we prove our main result Theorem~\ref{MT1}.

\section{A family of singular integral operators and the proof of \eqref{comder}}\label {Sec:2}
The main goal of this section is to establish the relations \eqref{comder}. To this end we first
 introduce a  family of multilinear singular integral operators which play a key role in the analysis of the unconfined Muskat problem and also of the unconfined quasistationary Stokes problem~\cite{MP2021}.
Given~${n,\,m\in\N}$,  Lipschitz continuous  functions ${a_1,\ldots, a_{m},\, b_1, \ldots, b_n:\mathbb{R}\to\mathbb{R}}$, and $\vartheta\in L_p(\R)$,  we set
\begin{equation}\label{BNM}
B_{n,m}(a_1,\ldots, a_m)[b_1,\ldots,b_n,\vartheta](x):=
\frac{1}{\pi}\PV\int_\mathbb{R}  \frac{\vartheta(x-y)}{y}\cfrac{\prod_{i=1}^{n}\big(\delta_{[x,y]} b_i\big) /y}{\prod_{i=1}^{m}\big[1+\big[\big(\delta_{[x,y]}  a_i\big) /y\big]^2\big]}\, dy,\quad x\in\R.
\end{equation}
If   ${f:\mathbb{R}\to\mathbb{R}}$  is Lipschitz continuous  we further set
\begin{equation}\label{defB0}
B^0_{n,m}(f)[\vartheta]:=B_{n,m}(f,\ldots, f)[f,\ldots,f,\vartheta].
\end{equation}

These operators have been consider in the $L_p$-setting with $p\in(1,\infty)$ in \cite{AM22, MM21}. We recall the following results.

\begin{lemma}\label{L:MP0} Let $p\in(1,\infty)$, $n,\,m \in\N$, and  $s\in(1+1/p,2)$. 
\begin{itemize}
\item[(i)] Let $a_1,\ldots, a_{m},\, b_1, \ldots, b_n:\R\to\R$ be Lipschitz continuous.
Then, there exists a positive constant~$C=C(n,\, m,\,\max_{i=1,\ldots, m}\|a_i'\|_{\infty} )$
such that
\[
\|B_{n,m}(a_1,\ldots, a_m)[b_1,\ldots,b_n,\,\cdot\,]\|_{\kL(L_p(\R))}\leq C\prod_{i=1}^{n} \|b_i'\|_{\infty}.
\]
 Moreover, $B_{n,m}\in {\rm C}^{1-}(W^1_\infty(\R)^{m},\kL^n_{\rm sym}(W^1_\infty(\R),\kL( L_p(\R)))).$\\[-1ex]
\item[(ii)] Given $a_1,\ldots, a_m \in W^s_p(\mathbb{R})$, there exists a constant~$C=C(n,\, m,\, s,\,  \max_{1\leq i\leq m}\|a_i\|_{W^s_p}$) such that
\begin{align} 
\| B_{n,m}(a_1,\ldots, a_{m})[b_1,\ldots, b_n,\vartheta]\|_{W^{s-1}_p}\leq C \|\vartheta\|_{W^{s-1}_p}\prod_{i=1}^{n}\|b_i'\|_{W^{s-1}_p} \label{REF1'}
\end{align}
for all $ b_1,\ldots, b_n\in W^s_p(\mathbb{R})$ and $\vartheta\in W^{s-1}_p(\mathbb{R}).$

Moreover,   $  B_{n,m}\in {\rm C}^{1-}(W^s_p(\mathbb{R})^m,\kL^n_{\rm sym}(W_p^{s}(\mathbb{R}),\kL(W^{s-1}_p(\mathbb{R})))).$

\item[(iii)]  Let  $n\geq1$. Given  $a_1,\ldots, a_m\in W^s_p(\mathbb{R})$, there exists 
$C=C(n,\, m,\, s,\,\max_{1\leq i\leq m}\|a_i\|_{W^s_p})$  such that
\begin{align} 
&\| B_{n,m}(a_1,\ldots, a_{m})[b_1,\ldots, b_n,\vartheta]\|_p\leq C\|b_1'\|_{p}\|\vartheta\|_{W^{s-1}_p}\prod_{i=2}^{n}\|b_i'\|_{W^{s-1}_p} \label{REF1}
\end{align}
for all $b_1,\ldots, b_n\in W^s_p(\mathbb{R})$ and $\vartheta\in W^{s-1}_p(\mathbb{R}).$

Moreover,   
$  B_{n,m}\in {\rm C}^{1-}(W^s_p(\mathbb{R})^m,\mathcal{L}^{n}( W^1_p(\mathbb{R})\times W_p^{s}(\mathbb{R})^{n-1},\kL( W^{s-1}_p(\mathbb{R}), L_p(\mathbb{R})))).$
\end{itemize} 
\end{lemma}
\begin{proof}
See \cite[Lemma 2]{AM22} for the proof of (i),  \cite[Lemma 5]{AM22} for the proof of (ii), and   \cite[Lemma 4]{AM22} for the proof of (iii).
\end{proof}

Before establishing \eqref{comder}, we prove in Lemma~\ref{L:MP1} below that $B_{n,m}$ maps into $W^1_p(\R)$ provided that its arguments are more regular. 
 The proof   uses the following algebraic  property
 \begin{equation}\label{rell}
 \begin{aligned}
 & \hspace{-1cm}\big(B_{n,m}(\wt a_{1}, \ldots, \wt a_{m}) -  B_{n,m}(a_1, \ldots, a_{m})\big)[b_1,\ldots, b_{n},\vartheta]\\[1ex]
 &=\sum_{i=1}^{m} B_{n+2,m+1}(\wt a_{1},\ldots, \wt a_{i},a_i,\ldots, a_{m})[b_1,\ldots, b_{n},a_i+\wt a_{i}, a_i-\wt a_{i},\vartheta].
\end{aligned}
 \end{equation}

\begin{lemma}\label{L:MP1}
 Let  $ n,\, m\in\N$,  $a_1,\ldots, a_m, \, b_1,\ldots,b_n\in W^2_p(\R),$ and   $\vartheta\in W^1_p(\R)$ be given. 
  The function~$ \varphi:=B_{n,m}(a_1,\ldots, a_{m})[b_1,\ldots, b_n,\vartheta]$ belongs then to $W^1_p(\R)$ and 
 \begin{equation}\label{FDER}
\varphi'(x)=\frac{1}{\pi}\PV\int_\mathbb{R}  \frac{\p}{\p x}\left(\frac{\vartheta(x-y)}{y}
\cfrac{\prod_{i=1}^{n}\big(\delta_{[x,y]} b_i\big) /y}{\prod_{i=1}^{m}\big[1+\big[\big(\delta_{[x,y]}  a_i\big) /y\big]^2\big]}\right)\, dy.
\end{equation}
 \end{lemma}
\begin{proof}. 
Recalling that the generator of the $C_0$-group $\{\tau_\xi\}_{\xi\in\R}\subset\kL(W^s_p(\R)),$ $s\geq0$,  
 is the linear operator~$[f\mapsto f']\in\kL(W^{s+1}_p(\R),W^s_p(\R)),$
 it suffices to show that  ${D_\xi\varphi:=(\tau_\xi\varphi-\varphi)/\xi}$ converges in~$L_p(\R)$ in the limit $\xi\to0$. 
 To this end we infer from~\eqref{rell} that 
\begin{align*}
 D_\xi\varphi&=\sum_{i=1}^nB_{n,m}(\tau_\xi a_1,\ldots,\tau_\xi a_m)\big[b_1,\ldots,b_{i-1}, D_\xi b_i,\tau_\xi b_{i+1},\ldots,\tau_\xi b_n,\tau_\xi\vartheta\big]\\[1ex]
&\hspace{0,45cm}+ B_{n,m}(\tau_\xi a_1,\ldots,\tau_\xi a_m)\big[b_1,\ldots,, b_n, D_\xi \vartheta\big]\\[1ex]
&\hspace{0,45cm}-\sum_{i=1}^mB_{n+2,m+1}(\tau_\xi a_1,\ldots,\tau_\xi a_i,a_i,\ldots,a_m)\big[b_1,\ldots,b_n, D_\xi a_i,\tau_\xi a_i+a_i, \vartheta\big].
\end{align*}
 In view of  Lemma \ref{L:MP0}~(i) and~(iii) we may pass  to the limit $\xi\to0$ in $L_2(\R)$ in  the latter  equality to conclude that $\varphi\in W^1_p(\R)$ with 
\begin{equation}\label{FDER'}
\begin{aligned}
\varphi'&=B_{n,m}(  a_1,\ldots,  a_m) [b_1,\ldots , b_n, \vartheta' ]\\[1ex]
&\hspace{0,45cm}+\sum_{i=1}^nB_{n,m}(a_1,\ldots,a_m)[b_1,\ldots,b_{i-1}, b_i',b_{i+1},\ldots  b_n,\vartheta]\\[1ex]
&\hspace{0,45cm}-2\sum_{i=1}^mB_{n+2,m+1}( a_i, a_1,\ldots,a_m) [b_1,\ldots,b_n, a_i',a_i, \vartheta].
\end{aligned}
\end{equation}
The relation \eqref{FDER} follows now directly from \eqref{FDER'} and  the definition \eqref{BNM} of the operator $B_{n,m}$.
\end{proof}

We are now in a position to prove \eqref{comder}.

\begin{prop}\label{Prop:1}
 Given $f\in W^{2}_p(\R)$  and $\vartheta\in W^1_p(\R)$, the functions $\bA(f)^*[\vartheta]$ and $\bB(f)^*[\vartheta]$ belong to $W^1_p(\R)$ and 
\begin{equation*}
(\bA(f)^*[\vartheta])'=-\bA(f)[\vartheta'] \qquad\text{and}\qquad (\bB(f)^*[\vartheta])'=-\bB(f)[\vartheta'].
\end{equation*} 
\end{prop}
\begin{proof}
The operators $\bA(f)^*$ and $\bB(f)^*$ are given by the formulas
\begin{equation}\label{adj}
\begin{aligned}
\bA(f)^*[\vartheta](x)&=\frac{1}{\pi}\PV\int_\R\frac{\vartheta(x-y)}{y}\frac{\big(\delta_{[x,y]}f\big)/y-f'(x-y)}{1+\big[\big(\delta_{[x,y]}f\big)/y\big]^2}\, dy,\\[1ex]
\bB(f)^*[\vartheta](x)&=-\frac{1}{\pi}\PV\int_\R\frac{\vartheta(x-y)}{y}\frac{1+f'(x-y)\big(\delta_{[x,y]}f\big)/y}{1+\big[\big(\delta_{[x,y]}f\big)/y\big]^2}\, dy,
\end{aligned}
\end{equation}
cf. \cite{MBV18}. 
Recalling \eqref{BNM} and \eqref{defB0}, we have 
\begin{equation}\label{adjBNM}
\begin{aligned}
&\bA(f)[\vartheta] =f'B_{0,1}^0(f)[\vartheta]-B_{1,1}^0(f)[\vartheta],\qquad&&\bA(f)^*[\vartheta] =B_{1,1}^0(f)[\vartheta]-B_{0,1}^0(f)[f'\vartheta],\\[1ex]
&\bB(f)[\vartheta] =B_{0,1}^0(f)[\vartheta]+f'B_{1,1}^0(f)[\vartheta],\qquad &&\bB(f)^*[\vartheta] =-B_{0,1}^0(f)[\vartheta]-B_{1,1}^0(f)[f'\vartheta].
\end{aligned}
\end{equation}
Since $W^1_p(\R)$ is an algebra, we immediately obtain from Lemma \ref{L:MP1} that the functions in~\eqref{adjBNM} belong all to~$W^1_p(\R)$.
Taking now advantage of the identity
\begin{align*}
\frac{\p}{\p x}\frac{\delta_{[x,y]}f-yf'(x-y)}{y^2+ \big(\delta_{[x,y]}f\big)^2}
=\frac{\p}{\p y}\frac{-y\big(\delta_{[x,y]}f'\big)}{y^2+ \big(\delta_{[x,y]}f\big)^2},\qquad x,\, y\in\R,\, y\neq 0, 
\end{align*}
Lemma~\ref{L:MP1}, \eqref{adjBNM}, and  integration by parts lead to
\begin{align*}
(\bA(f)^*[\vartheta])'(x)&=\bA(f)^*[\vartheta'](x)+\frac{1}{\pi}\PV\int_\R \vartheta(x-y)\frac{\p}{\p x}\frac{\delta_{[x,y]}f-yf'(x-y)}{y^2+ \big(\delta_{[x,y]}f\big)^2}\, dy\\[1ex]
&=\bA(f)^*[\vartheta'](x)-\frac{1}{\pi}\PV\int_\R \vartheta(x-y)\frac{\p}{\p y}\frac{y\big(\delta_{[x,y]}f'\big)}{y^2+ \big(\delta_{[x,y]}f\big)^2}\, dy\\[1ex]
&=\bA(f)^*[\vartheta'](x)-\frac{1}{\pi}\PV\int_\R \vartheta'(x-y)\frac{y\big(\delta_{[x,y]}f'\big)}{y^2+ \big(\delta_{[x,y]}f\big)^2}\, dy\\[1ex]
&=\big(\bA(f)^*[\vartheta']-f'B_{0,1}^0(f)[\vartheta']+B_{0,1}^0(f)[f'\vartheta']\big)(x)\\[1ex]
&=-\bA(f)[\vartheta'](x) 
\end{align*}
for almost all $x\in\R$.
Arguing similarly, we obtain in view of the relation 
\begin{align*}
\frac{\p}{\p x}\frac{y+f'(x-y)\big(\delta_{[x,y]}f\big)}{y^2+\big(\delta_{[x,y]} f\big)^2}
=\frac{\p}{\p y}\frac{ \big(\delta_{[x,y]}f\big)\big(\delta_{[x,y]}f'\big)}{y^2+\big(\delta_{[x,y]} f\big)^2 },\qquad x,\, y\in\R,\, y\neq 0, 
\end{align*}
that 
\begin{align*}
(\bB(f)^*[\vartheta])'(x)&=\bB(f)^*[\vartheta'](x)-\frac{1}{\pi}\PV\int_\R \vartheta(x-y)\frac{\p}{\p x}\frac{y+f'(x-y)\big(\delta_{[x,y]}f\big)}{y^2+\big(\delta_{[x,y]} f\big)^2}\, dy\\[1ex]
&=\bB(f)^*[\vartheta'](x)-\frac{1}{\pi}\PV\int_\R \vartheta(x-y)\frac{\p}{\p y}\frac{ \big(\delta_{[x,y]}f\big)\big(\delta_{[x,y]}f'\big)}{y^2+\big(\delta_{[x,y]} f\big)^2 }\, dy\\[1ex]
&=\bB(f)^*[\vartheta'](x)-\frac{1}{\pi}\PV\int_\R \vartheta'(x-y)\frac{ \big(\delta_{[x,y]}f\big)\big(\delta_{[x,y]}f'\big)}{y^2+\big(\delta_{[x,y]} f\big)^2 }\, dy\\[1ex]
&=\big(\bB(f)^*[\vartheta']-f'B_{1,1}^0(f)[\vartheta']+B_{1,1}^0(f)[f'\vartheta']\big)(x)\\[1ex]
&=-\bB(f)[\vartheta'](x) 
\end{align*}
for almost all $x\in\R$. This completes the proof.
\end{proof}

\section{On the spectrum of $\bA(f)^*$ and the equivalent formulation~\eqref{P:2} of~\eqref{P:1}}\label {Sec:3}

The core of the analysis in this section addresses the invertibility of the operator $\lambda-\bA(f)^*$ with~${\lambda\in\R\setminus(-1,1)}$.
This property is used below when establishing the equivalence of the two formulations \eqref{P:1} and \eqref{P:2}, but also in Section~\ref{Sec:4} when formulating \eqref{P:2} as an evolution problem for $f$. 

Let~$p\in(1,\infty)$ and $s\in (1+1/p,2)$ be fixed in the following.
We provide three invertibility results.
Given~${f\in W^s_p(\R)}$, we show that $\lambda-\bA(f)^* $ belongs to $ {\rm Isom}(L_p(\R))$
 and   ${\rm Isom}(W^{s-1}_p(\R))$  for each~${\lambda\in\R\setminus(-1,1)}$,  see Theorem~\ref{Thm:1} and Proposition~\ref{Prop:2} below. 
 Moreover, under the assumption that~${f\in W^2_p(\R)}$,  we  prove additionally that    $\lambda-\bA(f)^*\in {\rm Isom}(W^{1}_p(\R))$, see Corollary~\ref{Cor:1}.

These invertibility properties together with the the corresponding invertibility results for $\lambda-\bA(f)$ established in \cite{AM22} immediately provide the equivalence of the formulations 
\eqref{P:1} and \eqref{P:2} in the setting of solutions which satisfy $f(t)\in W^3_p(\R)$ for all~$t>0$ (our solutions all share this regularity, see Theorem~\ref{MT1}). 
 Indeed, Corollary~\ref{Cor:1}, the relation~\eqref{comder}, \cite[Theorem 3 and Theorem 4]{AM22}, and the fact that~$a_\mu\in(-1,1)$ combined show  that, given~${f\in W^3_p(\R)}$, the unique solution $\vartheta\in W^1_p(\R)$ to
\[
(1+a_\mu\bA(f)^*)[\vartheta]=b_\mu \big(\sigma\kappa(f)-\Theta f\big)
\]
has the property that $\oo:=\vartheta'\in L_p(\R)$ is the unique solution to
\[
(1-a_\mu\bA(f))[\oo]=b_\mu \big(\sigma\kappa(f)-\Theta f\big)'.
\]
The equivalence of \eqref{P:1} and \eqref{P:2} is now a direct consequence of this  property and of~\eqref{comder}.
 
It remains to establish the invertibility results mentioned above.
The main step  is to prove the invertibility in $\kL(L_p(\R))$, see Theorem~\ref{Thm:1} below. 
When establishing this  property in~$\kL(W^{s-1}_p(\R))$ and~$\kL(W^{1}_p(\R))$ we use to a large extent   this result. 

\begin{thm}\label{Thm:1}
Given  $p\in(1,\infty)$, $s\in (1+1/p,2)$, and $f\in W^s_p(\R)$, we have
\[
\lambda-\bA(f)^* \in{\rm Isom}(L_p(\R))\qquad\text{ for all~$\lambda\in\R\setminus(-1,1)$.}
\]
\end{thm}

The proof of   Theorem~\ref{Thm:1} is presented later on in this section  as it necessitates some preparation.
In the case of a bounded Lipschitz domain, the invertibility issue for the  layer potentials
for Laplace’s equation has been addressed in \cite{Ve84}. 
In the  unbounded graph geometry considered herein,  it was recently shown in \cite{AM22}, by using different arguments than in \cite{Ve84}, that, under the hypotheses of Theorem~\ref{Thm:1},
 $\lambda-\bA(f)\in{\rm Isom}(L_p(\R))$, see \cite[Theorem 3 and Theorem 4]{AM22}.
Theorem~\ref{Thm:1} is in the particular case~${p\in(1,2]}$ a direct consequence of these results. 
Indeed, if $p\in(1,2]$, Sobolev's embedding ensures that~$f\in W^s_p(\R)\hookrightarrow W^{s'}_{p'}(\R)$, where $p'$ is the adjoint exponent to $p$, that is~${1/p+1/p'=1}$, and~$s'=s-1/p+1/p'\in (1+1/p',2).$
 \cite[Theorem~3 and Theorem~4]{AM22}  then ensure that~${\lambda-\bA(f)}\in{\rm Isom}(L_{p'}(\R)),$ and, by duality, we conclude that~${\lambda-\bA(f)^*\in {\rm Isom}(L_p(\R)).}$
This argument is clearly valid only when $p\in(1,2]$. 
For $p>2$ we argue differently, the   strategy of proof being similar as that of \cite[Theorem 3 and Theorem 4]{AM22}.
To start, we  choose  for each~${\varepsilon\in(0,1)}$, a finite $\varepsilon$-localization family, that is  a family  
\[\{(\pi_j^\varepsilon, x_j^\e)\,:\, -N+1\leq j\leq N\}\subset  {\rm BUC}^\infty(\mathbb{R},[0,1])\times\R,\]
with $N=N(\varepsilon)\in\mathbb{N} $ sufficiently large, such that $x_j^\e\in\supp\pi_j^\e$, $-N+1\leq j\leq N$, and
\begin{align*}
\bullet\,\,\,\, \,\,  & \mbox{$\supp \pi_j^\varepsilon \subset\{|x|\leq \varepsilon+1/\varepsilon\}$ is an interval of length $\varepsilon$ for  $|j|\leq N-1$, $\supp \pi_{N}^\varepsilon\subset\{|x|\geq 1/\e\}$;} \\[1ex]
\bullet\,\,\,\, \,\, &\mbox{ $ \pi_j^\varepsilon\cdot  \pi_l^\varepsilon=0$ if $[|j-l|\geq2, \max\{|j|, |l|\}\leq N-1]$ or $[|l|\leq N-2, j=N];$} \\[1ex]
\bullet\,\,\,\, \,\, &\mbox{ $\displaystyle\sum_{j=-N+1}^N(\pi_j^\varepsilon)^2=1;$} \\[1ex]
 \bullet\,\,\,\, \,\, &\mbox{$\|(\pi_j^\varepsilon)^{(k)}\|_\infty\leq C\varepsilon^{-k}$ for all $ k\in\mathbb{N}, -N+1\leq j\leq N$.} 
\end{align*} 
To each finite $\varepsilon$-localization family we associate  a second family   
$$\{\chi_j^\varepsilon\,:\, -N+1\leq j\leq N\}\subset {\rm BUC}^\infty(\mathbb{R},[0,1])$$ with the following properties
\begin{align*}
\bullet\,\,\,\, \,\,  &\mbox{$\chi_j^\varepsilon=1$ on $\supp \pi_j^\varepsilon$;} \\[1ex]
\bullet\,\,\,\, \,\,  &\mbox{$\supp \chi_j^\varepsilon$ is an interval  of length $3\varepsilon$ and with the same midpoint as $ \supp \pi_j^\varepsilon$, $|j|\leq N-1$;} \\[1ex]
\bullet\,\,\,\, \,\,  &\mbox{$\supp\chi_N^\varepsilon\subset \{|x|\geq 1/\varepsilon-\varepsilon\}$ and $\xi+ \supp \pi_{N}^\varepsilon\subset\supp \chi_{N}^\varepsilon$ for $|\xi|<\varepsilon.$}  
\end{align*} 
As stated in \cite[Lemma 9]{AM22}, the norm
\[
\Big[f\mapsto \sum_{j=-N+1}^N\|\pi_j^\varepsilon f\|_{W^r_p}\Big]:W^r_p(\R)\to[0,\infty),\qquad p\in(1,\infty),\, r\geq 0,
\]
 is equivalent to the standard norm on $W^r_p(\R)$.

In order to prove the injectivity of  $\lambda-\bA(f)^*,$ we need to establish the unique solvability of the equation
\begin{equation}\label{solv}
(\lambda-\bA(f)^*)[\vartheta]=0
\end{equation}
in $L_p(\R)$. 
Since \eqref{solv} is equivalent to the system
\begin{equation}\label{syssolv}
(\lambda-\chi_j^\e\bA(f)^*\chi_j^\e)[\pi_j^\e\vartheta]=\chi_j^\e(\pi_j^\e\bA(f)^*[\vartheta]-\bA(f)^*[\pi_j^\e\vartheta]),\qquad -N+1\leq j\leq N,
\end{equation}
we next consider the operators on the  right  of \eqref{syssolv} and prove they are all compact.

\begin{lemma}\label{L:comp}
Let   $p\in[2,\infty)$, $s\in(1+1/p,2)$, and $f\in W^s_p(\mathbb{R})$.
Given ${-N+1\leq j\leq N}$, the linear operator $K_j:L_p(\mathbb{R})\to L_p(\mathbb{R})$ defined by
\begin{align}\label{Kjtheta}
K_j[\vartheta]:=\chi_{j}^\varepsilon(\pi_{j}^\varepsilon\mathbb{A}(f)^*[ \vartheta]-\mathbb{A}(f)^*[\pi_{j}^\varepsilon \vartheta]),\qquad \vartheta\in L_p(\mathbb{R}),
\end{align}
is compact.
\end{lemma}
\begin{proof}  
According to the Riesz-Fr\'echet-Kolmogorov theorem, it suffices to show that
\begin{align}\label{prop12}
\sup_{\|\vartheta\|_p\leq 1}\int_{\{|x|>R\}}|K_j[\vartheta](x)|^p\, dx\underset{R\to\infty}\to0\qquad\text{and}\qquad\sup_{\|\vartheta\|_p\leq 1}\|\tau_\xi (K_j[\vartheta])-K_j[\vartheta]\|_p\underset{\xi\to0}\to0.
\end{align}

\noindent{Step 1.} In order to prove the first convergence in \eqref{prop12}, we set $a_\e:=\e+1/\e$ and we note that $\pi_N^\e=1$ for~${|x|> a_\e}$.
Let $R>2a_\e$. We then have $\chi_j^\e(x)=0$ for all $|x|>R $ and~$|j|\leq N-1$.
Hence, it remains to establish the convergence for~$j=N$.
Since for $|y|<R/2$ we have $|x-y|\geq |x|-R/2> a_\e,$ it holds that $\delta_{[x,y]}\pi_\e^N=0$ for all~$|x|>R$ and~${|y|<R/2,}$ and therefore
\begin{align*}
&\hspace{-0,5cm}\Big(\int_{\{|x|>R\}}|K_N[\vartheta](x)|^p\, dx\Big)^{1/p}\\[1ex]
&\leq \Big(\int_{\{|x|>R\}}
\Big(\int_{\{|y|>R/2\}}\Big|\vartheta(x-y)\frac{\delta_{[x,y]}f-yf'(x-y)}{y^2}\delta_{[x,y]}\pi_N^\e\Big|\, dy\Big)^p\, dx\Big)^{1/p}\\[1ex]
&\leq T_1+T_2+T_3,
\end{align*}
where, in view of $|\delta_{[x,y]}\pi_\e^N|\leq 1 $ and $\|\vartheta\|_p\leq 1$, we have
\begin{align*}
T_1&:=\Big(\int_{\{|x|>R\}}|f(x)|^p\Big(\int_{\{|y|>R/2\}}\frac{|\vartheta(x-y)|}{y^2}\, dy\Big)^p\, dx\Big)^{1/p}\leq \frac{C}{R^{(p+1)/p}}\|f\|_p\underset{R\to\infty}\to0,\\[1ex]
T_2&:=\Big(\int_{\{|x|>R\}}\Big(\int_{\{|y|>R/2\}}\frac{|(f\vartheta)(x-y)|}{y^2}\, dy\Big)^p\, dx\Big)^{1/p}\\[1ex]
&\,\leq\int_{\{|y|>R/2\}}\frac{1}{y^2}\Big(\int_{\{|x|>R\}}|(f\vartheta)(x-y)|^p\, dx\Big)^{1/p}\, dy\leq \frac{C}{R}\|f\|_\infty\underset{R\to\infty}\to0,\\[1ex]
T_3&:=\Big(\int_{\{|x|>R\}}\Big(\int_{\{|y|>R/2\}}\frac{|(f'\vartheta)(x-y)|}{|y|}|\delta_{[x,y]}\pi_\e^N|\, dy\Big)^p\, dx\Big)^{1/p}.
\end{align*}
We used  H\"older's inequality to estimate $T_1$ and Minkowski's integral inequality when we considered~$T_2$.
With respect to $T_3$ we note that, given $|x|>R,$ the term $\delta_{[x,y]}\pi_\e^N$ can be different from $0$ only if   $y\in (x-a_\e,x+a_\e)$.
Using H\"older's inequality,  we then get 
\begin{align*}
T_{3}&\leq \Big(\int_{\{|x|>R\}}\Big(\int_{x-a_\e}^{x+a_\e}\frac{|(f'\vartheta)(x-y)|}{|y|}\, dy\Big)^p\, dx\Big)^{1/p}\\[1ex]
&\leq\|f'\vartheta\|_p \Big(\int_{\{|x|>R\}}\Big(\int_{x-a_\e}^{x+a_\e}\frac{1}{|y|^{p'}}\, dy\Big)^{p/p'}\, dx\Big)^{1/p}\\[1ex]
&\leq  C\|f'\|_\infty \Big(\int_{\{|x|>R\}}\Big|\frac{1}{|x-a_\e|^{p'-1}}-\frac{1}{|x+a_\e|^{p'-1}}\Big|^{p/p'}\, dx\Big)^{1/p},
\end{align*}
where, as before, $p'\in (1,2]$ is the adjoint exponent to $p$.
Since $p'-1\in(0,1]$, the inequality~${a^r-b^r\leq (a-b)^r}$, which holds for all $0<b<a$ and $r\in(0,1)$, together with the estimate~${|x\pm a_\e|\geq |x|/2}$ for  $|x|>R$ leads us to
\begin{align*}
T_{3}&\leq  C\|f'\|_\infty \Big(\int_{\{|x|>R\}}\frac{1}{x^{2}} \, dx\Big)^{1/p}=\frac{C}{R}\|f'\|_\infty \underset{R\to\infty}\to0.
\end{align*}

\noindent{Step 2.} We now establish the second convergence in \eqref{prop12}. Let therefore~${\|\vartheta\|_p\leq 1}$
 and $\xi\in\R$ be arbitrary with $|\xi|<\min\{\e,\, 1/2\}$. 
 Given~${-N+1\leq j\leq N}$, we have
\[
\|\tau_\xi (K_j[\vartheta])-K_j[\vartheta]\|_p\leq \widetilde T_1+\widetilde T_2+\widetilde T_3,
\]
where, after a suitable change of variables,  the right side of the latter inequality may be expressed~as
\begin{align*}
\widetilde T_1&:=\|\tau_\xi\chi_j^\e-\chi_j^\e\|_\infty\|\pi_{j}^\varepsilon\mathbb{A}(f)^*[ \vartheta]-\mathbb{A}(f)^*[\pi_{j}^\varepsilon \vartheta]\|_p\leq C|\xi|\underset{\xi\to0}\to0,\\[1ex]
\widetilde T_2&:=\Big(\int_\R\Big(\int_\R K_1^\e(x,y,\xi)(f'\vartheta)(x-y)\, dy  \Big)^p\, dx\Big)^{1/p},\\[1ex]
\widetilde T_3&:=\Big(\int_\R\Big(\int_\R  K_2^\e(x,y,\xi)\vartheta (x-y)\, dy  \Big)^p\, dx\Big)^{1/p},
\end{align*}
with
\begin{align*}
K_1^\e(x,y,\xi)&:=\frac{\big(\delta_{[x+\xi,y+\xi]}\pi_j^\e\big)/(y+\xi)}{1+\big[\big(\delta_{[x+\xi,y+\xi]}f\big)/(y+\xi)\big]^2}
-\frac{\big(\delta_{[x,y]}\pi_j^\e\big)/y}{1+\big[\big(\delta_{[x,y]}f\big)/y\big]^2},\\[1ex]
K_2^\e(x,y,\xi)&:=\frac{\big[\big(\delta_{[x+\xi,y+\xi]}f\big)/(y+\xi)\big]\cdot\big[\big(\delta_{[x+\xi,y+\xi]}\pi_j^\e\big)/(y+\xi)\big]}{1+\big[\big(\delta_{[x+\xi,y+\xi]}f\big)/(y+\xi)\big]^2}
-\frac{\big[\big(\delta_{[x,y]}f\big)/y\big]\cdot\big[\big(\delta_{[x,y]}\pi_j^\e\big)/y\big]}{1+\big[\big(\delta_{[x,y]}f\big)/y\big]^2}.
\end{align*}
Let now $a_\e:=\e+1/\varepsilon+1/2$. Then, ${\rm supp \,}\pi_j^\e$,~${|j|\leq N-1}$, as well as
  ${\rm supp \,}(1-\pi_N^\e)$ are subintervals of~${\{|x|\leq a_\e-1/2\}}$.
Since $|\xi|< 1/2$ and $W^s_p(\R)\hookrightarrow {\rm BUC}^{s-1/p}(\R)$, we have
\begin{equation}\label{2est}
\begin{aligned}
&\Big|\frac{\delta_{[x+\xi,y+\xi]}\pi_j^\e}{y+\xi}-\frac{\delta_{[x,y]}\pi_j^\e}{y}\Big|
\leq C|\xi|\Big[\Big(\frac{1}{|y|}{\bf 1}_{\{|x|\leq a_\e\}}(x)+\frac{1}{y^2}\Big){\bf 1}_{\{|y|>1\}}(y)+{\bf 1}_{\{|y|<1\}}(y)\Big],\\[1ex]
&\Big|\frac{\delta_{[x+\xi,y+\xi]}f}{y+\xi}-\frac{\delta_{[x,y]}f}{y}\Big|
\leq C|\xi|^{s-1-1/p} \Big(\frac{1}{|y|} {\bf 1}_{\{|y|>1\}}(y)+{\bf 1}_{\{|y|<1\}}(y)\Big)
\end{aligned}
\end{equation}
for $x,\, y\in\R$ and $-N+1\leq j\leq N$.
 The  relations \eqref{2est} lead us to
\begin{align*}
|K_1^\e(x,y,\xi)|&\leq C|\xi|^{s-1-1/p}\Big(\frac{1}{|y|}{\bf 1}_{\{|x|\leq a_\e\}}(x){\bf 1}_{\{|y|>1\}}(y)+\frac{1}{y^2}{\bf 1}_{\{|y|>1\}}(y)+{\bf 1}_{\{|y|<1\}}(y)\Big),\\[1ex]
|K_2^\e(x,y,\xi)|&\leq C|\xi|^{s-1-1/p}\Big(\frac{1}{y^2}{\bf 1}_{\{|y|>1\}}(y)+{\bf 1}_{\{|y|<1\}}(y)\Big).
\end{align*}
H\"older's inequality and Minkowski's inequality  lead, in view of the latter estimates, to
\begin{align*}
&\Big(\int_{\{|x|\leq a_\e\}}\Big(\int_{\{|y|>1\}}\frac{|(f'\vartheta)(x-y)|}{|y|}\, dy  \Big)^p\, dx\Big)^{1/p}
\leq \|f'\vartheta\|_p\Big(\int_{\{|y|>1\}}\frac{1}{|y|^{p'}}\, dy  \Big)^{1/p'}(2a_\e)^{1/p}\leq C\|f'\|_\infty,\\[1ex]
&\Big(\int_{\R}\Big(\int_{\{|y|>1\}}\frac{1}{y^2}|(f'^{k}\vartheta)(x-y)|\, dy  \Big)^p\, dx\Big)^{1/p}
\leq \|f'^k\vartheta\|_p\Big(\int_{\{|y|>1\}}\frac{1}{|y|^{2}}\, dy  \Big)^{1/p}\leq C\|f'\|_\infty^k, \quad k=0,\, 1,\\[1ex]
&\Big(\int_{\R}\Big(\int_{\{|y|<1\}} |(f'^k\vartheta)(x-y)|\, dy  \Big)^p\, dx\Big)^{1/p}
\leq \|f'^k\vartheta\|_p\Big(\int_{\{|y|<1\}}1\, dy  \Big)^{1/p}\leq C\|f'\|_\infty^k, \quad k=0,\, 1,
\end{align*}
and we conclude that 
\[
\widetilde T_2+\widetilde T_3\leq C|\xi|^{s-1-1/p}\underset{\xi\to0}\to 0,
\]
which completes our arguments.
\end{proof}

In the next lemma we consider the left side of \eqref{syssolv} and prove that, if $\e$ is chosen sufficiently small, the operator on the left of \eqref{syssolv} is invertible 
in $\mathcal{L}(L_p(\mathbb{R}))$ for all~$\lambda\in\R\setminus(-1,1).$
\begin{lemma}\label{L:comm'}
Let   $p\in[2,\infty)$, $s\in(1+1/p,2)$, and $f\in W^s_p(\mathbb{R})$.
If $\e$ is sufficiently small, then
\begin{align*}
\|\chi_j^\varepsilon\mathbb{A}(f)^*\chi_j^\varepsilon\|_{\mathcal{L}(L_p(\mathbb{R}))}<1\qquad \text{ for all ${-N+1\leq j\leq N}$.}
\end{align*} 
\end{lemma}
\begin{proof}
\noindent{Step 1.}
We first establish  the estimate  for $|j|\leq N-1$.
Using Minkowski's inequality, the mean value theorem, and the embedding $W^s_p(\mathbb{R})\hookrightarrow {\rm BUC}^{s-1/p}(\mathbb{R}),$ we get in view of $\eqref{adj}_1$
\begin{align*}
\|\chi_j^\varepsilon\mathbb{A}(f)^*[\chi_j^\varepsilon\vartheta]\|_p
&\leq\Big(\int_\mathbb{R}\Big(\int_\mathbb{R}\chi_j^\varepsilon(x)\frac{(\chi_j^\varepsilon\vartheta)(x-y)}{y}\frac{\big(\delta_{[x,y]}f\big)/y-f'(x-y)}{1+\big[\big(\delta_{[x,y]}f\big)/y\big]^2}\, dy\Big)^{p}\, dx\Big)^{1/p}\\[1ex]
&\leq [f']_{s-1-1/p}\int_\mathbb{R}|y|^{s-2-1/p}\Big(\int_{\supp \chi_j^\varepsilon}  |(\chi_j^\varepsilon\vartheta)(x-y)|^p \, dx\Big)^{1/p}\, dy\\[1ex]
&\leq C\|f\|_{W^s_p}\|\vartheta\|_p\int_{\supp \chi_j^\varepsilon-\supp \chi_j^\varepsilon}|y|^{s-2-1/p} \, dy\\[1ex]
&\leq C\|f\|_{W^s_p}\|\vartheta\|_p\varepsilon^{s-1-1/p}.
\end{align*}
Hence, if $\e$ is sufficiently small, then $\|\chi_j^\varepsilon\mathbb{A}(f)^*\chi_j^\varepsilon\|_{\mathcal{L}(L_p(\mathbb{R}))}<1$.\medskip

\noindent{Step 2.} Let now $j=N.$ Since~$\supp\chi_N^\varepsilon\subset \{|x|\geq 1/\varepsilon-1\}$ we have~${\|\chi_N^\varepsilon\mathbb{A}(f)[\chi_N^\varepsilon\vartheta]\|_p\leq T_1+T_2}$, where 
\begin{align*}
T_1&:=\Big(\int_{\{|x|>1/\e-1\}}\Big(\int_{\{|y|<1\}} \chi_N^\e(x)\frac{\delta_{[x,y]} f-yf'(x-y)}{1+\big(\delta_{[x,y]} f/y\big)^2}\frac{(\chi_N^\e\vartheta)(x-y)}{y^2}\, dy\Big)^{p}\, dx\Big)^{1/p},\\[1ex]
T_2&:=\Big(\int_{\{|x|>1/\e-1\}}\Big(\int_{\{|y|>1\}} \chi_N^\e(x)\frac{ \delta_{[x,y]} f-yf'(x-y)}{1+\big(\delta_{[x,y]} f/y\big)^2}
\frac{(\chi_N^\e\vartheta)(x-y)}{y^2}\, dy\Big)^{p}\, dx\Big)^{1/p}.
\end{align*}
With respect to $T_1$ we note that for  $|x|> 1/\varepsilon-1$ and $|y|<1$ the mean value theorem implies that 
\[
|\delta_{[x,y]} f-yf'(x-y)|\leq 2 \|f'\|_{L_\infty(\{|x|>1/\varepsilon- 2\})}^{1/2}[f']_{s-1-1/p}^{1/2}|y|^{s/2+1/2-1/2p},
\]
and, using  Minkowski's integral inequality, we get
\begin{align*}
T_1&\leq 2\|f'\|_{L_\infty([|x|>1/\varepsilon- 2])}^{1/2}[f']_{s-1-1/p}^{1/2}
\Big(\int_\mathbb{R}\Big(\int_{\{|y|<1\}} \frac{|(\chi_N^\e\vartheta)(x-y)|}{|y|^{3/2+1/2p-s/2}}\, dy\Big)^{p}\, dx\Big)^{1/p}\\[1ex]
&\leq  C\|f\|_{W^s_p}^{1/2}\|f'\|_{L_\infty(\{|x|>1/\varepsilon- 2\})}^{1/2} \int_{\{|y|<1\}} \frac{1}{|y|^{3/2+1/2p-s/2}}\Big(\int_{\mathbb{R}}| (\chi_N^\e\vartheta)(x-y)|^p \, dx\Big)^{1/p}\, dy\\[1ex]
&\leq C\|f\|_{W^s_p}^{1/2}\|f'\|_{L_\infty(\{|x|>1/\varepsilon- 2\})}^{1/2}\|\vartheta\|_p.
\end{align*}
In order to estimate $T_2$, we note that    $T_2\leq  T_{2a}+T_{2b}+T_{2c} $, where 
\begin{align*}
T_{2a}&:= \Big(\int_{\{|x|>1/\e-1\}}|f(x)|^p\Big(\int_{\{|y|>1\}}|\vartheta(x-y)|\, dy\Big)^{p}\, dx\Big)^{1/p}\leq C\|f\|_{L_p(\{|x|>1/\varepsilon-1\})}\|\vartheta\|_p,\\[1ex]
T_{2b}&:= \Big(\int_\R \Big(\int_{\{|y|>1\}}\frac{|f\chi_N^\e\vartheta(x-y)|}{y^2}\, dy\Big)^{p}\, dx\Big)^{1/p}
\leq \int_{\{|y|>1\}}\frac{1}{y^2}\Big(\int_{\mathbb{R}} |  (f\chi_N^\e\vartheta)(x-y)|^p \, dx\Big)^{1/p}\, dy\\[1ex]
&\,\leq C\|f\chi_N^\e\vartheta\|_p\,\leq C\|f\|_{L_\infty(\{|x|>1/\varepsilon-1\})}\|\vartheta\|_p,\\[1ex]
  T_{2c}&:=\Big(\int_{\R}\Big(\int_{\{|y|>1\}}  \frac{1}{1+\big(\delta_{[x,y]} f/y\big)^2}
\frac{(f'\chi_N^\e\vartheta)(x-y)}{y}\, dy\Big)^{p}\, dx\Big)^{1/p}.
\end{align*}
We used  H\"older's inequality to estimate $T_{2a}$ and Minkowski's integral inequality for~$T_{2b}$.
With respect to $T_{2c}$, a simple algebraic manipulation reveals that
\begin{align*}
  T_{2c}&\leq \Big(\int_{\R}\Big(\int_{\{|y|>1\}}  \frac{(\delta_{[x,y]} f)^2}{y^2}
\frac{(f'\chi_N^\e\vartheta)(x-y)}{y}\, dy\Big)^{p}\, dx\Big)^{1/p}+\pi\|H_1[f'\chi_N^\e\vartheta]\|_p,
\end{align*}
where $H_1$ denotes the truncated Hilber transform
\[
H_1[\vartheta](x)=\frac{1}{\pi}\int_{\{|y|>1\}} \frac{\vartheta(x-y)}{y}\, dy.
\]
Similarly as in the case of $T_{2b}$ we get
\begin{align*}
  \Big(\int_{\R}\Big(\int_{\{|y|>1\}}  \frac{(\delta_{[x,y]} f)^2}{y^2}
\frac{(f'\chi_N^\e\vartheta)(x-y)}{y}\, dy\Big)^{p}\, dx\Big)^{1/p}\leq C\|f\|_\infty^2\|f'\|_{L_\infty(\{|x|>1/\varepsilon-1\})}\|\vartheta\|_p. 
\end{align*}
Moreover,  it is well-known that  $H_1\in\kL(L_p(\R))$, hence
\[
\|H_1[f'\chi_N^\e\vartheta]\|_p\leq C\|f'\chi_N^\e\vartheta\|_{p} \leq C\|f'\|_{L_\infty(\{|x|>1/\varepsilon-1\})} \|\vartheta\|_{p}.
\]
Due to the fact that  $f\in W^s_p(\R)$, $s>1+1/p$,  for~$\e\to0$  we have ~${\|f^{(k)}\|_{L_\infty(\{|x|>1/\varepsilon-2\})}\to0}$,~${k=0, 1},$ and~${\|f\|_{L_p(\{|x|>1/\varepsilon-1\})}\to0}$.
The estimates established above now ensure that the claim  holds true also for~$j=N$.
\end{proof}

We are now in a position to establish the aforementioned invertibility property in $\kL(L_p(\R))$.

\begin{proof}[Proof of Theorem \ref{Thm:1}]
As mentioned in the  discussion following Theorem \ref{Thm:1}, it remains to consider the case  when $p\in(2,\infty)$. 
We devise the proof in two steps.\medskip

\noindent{Step 1.} In this step we show that $\lambda-\bA(f)^*$ is injective.
To start, we fix for each $\e\in(0,1)$ a function~${a_\e\in{\rm BUC}^\infty(\R,[0,1])}$ such that $a_\e=0$ in $\{|x|<1/\e\}$, $a_\e=1$ in $\{|x|>1/\e+1\},$ and~${\|a_\e'\|_\infty\leq2.}$
Recalling \eqref{adjBNM}, we may decompose $\bA(f)^*$ as the sum
\[
\bA(f)^*=\bA_1^\e+\bA_{2}^\e,
\]
where
\begin{align*}
\bA_1^\e&:=B_{1,1}(f)[a_\e f,\cdot]-B_{0,1}(f)[(a_\e f)'\cdot],\\[1ex]
\bA_2^\e&:=B_{1,1}(f)[(1-a_\e)f,\cdot]-B_{0,1}(f)[((1-a_\e) f)'\cdot].
\end{align*}
  Lemma~\ref{L:MP0}~(i) ensures  that $\bA_{i}^\e\in\kL(L_q(\R))$, $i=1,\, 2$, for all $q\in(1,\infty)$ and
\begin{equation*}
\|\bA_1^\e\|_{\kL(L_q(\R))}\leq C\|(a_\e f)'\|_\infty\leq C\|f\|_{W^1_\infty(\{|x|>1/\e\})}\underset{\e\to0}\to0.
\end{equation*}
Consequently, for sufficiently small $\e$ we have
\begin{equation}\label{isomp2}
\lambda-\bA_1^\e\in {\rm Isom}(L_2(\R))\cap {\rm Isom}(L_p(\R))\qquad \text{for all $\lambda\in\R\setminus(-1,1)$.}
\end{equation}

Let now $\lambda\in\R\setminus(-1,1)$ be fixed and let $\vartheta\in L_p(\R)$ be a solution of $(\lambda-\bA(f)^*)[\vartheta]=0,$ or equivalently
\[
(\lambda-\bA_1^\e)[\vartheta]=\bA_2^\e[\vartheta],
\] 
where $\e\in(0,1/4)$ is fixed such that \eqref{isomp2} holds true.
Our goal is to prove that $\vartheta\in L_2(\R)$, which in view of~\eqref{isomp2} is equivalent to showing that~${\bA_2^\e[\vartheta]\in L_2(\R)}$. 
Then, since $ \lambda-\bA(f)^*\in {\rm Isom}(L_2(\R))$ for all~${f\in{\rm BUC}^1(\R)}$ and $\lambda\in\R\setminus(-1,1)$, see the arguments in the proof of of \cite[Theorem 3.5]{MBV18}, 
we conclude that  $\vartheta=0$, hence $\lambda-\bA(f)^*\in\kL(L_p(\R))$ is injective.

In order to show that ${\bA_2^\e[\vartheta]\in L_2(\R)}$ we write 
\[
\bA_2^\e[\vartheta]=\bA_2^\e[a_{\e^2}\vartheta]+\bA_2^\e[(1-a_{\e^2})\vartheta],
\]
where, due to the fact that  $(1-a_{\e^2})=0$ in $\{|x|>1/\e^2+1\}$  and $\vartheta\in L_p(\R)$ with $p\in(2,\infty)$, we have~${(1-a_{\e^2})f'\vartheta\in L_2(\R)}$.
Lemma~\ref{L:MP0}~(i) now yields $\bA_2^\e[(1-a_{\e^2})\vartheta]\in L_2(\R)$, and it remains to prove that~$\bA_2^\e[a_{\e^2}\vartheta]\in L_2(\R)$.
To this end we set $g_\e:=(1-a_\e)f$ and note that 
\[
\|\bA_2^\e[a_{\e^2}\vartheta]\|_2\leq\Big(\int_{\R}\Big(\int_{\{|y|>1\}} \frac{(a_{\e^2}\vartheta)(x-y)}{y}\frac{\big(\delta_{[x,y]} g_\e/y\big)-g_\e'(x-y)}{1+\big(\delta_{[x,y]} f/y\big)^2}\, dy\Big)^{2}\, dx\Big)^{1/2}.
\]
Indeed,   $\e\in(0,1/4)$ implies that $1/\e^2-1>1/\e+2$. 
Consequently, if $x\in\R$, then either $|x|>1/\e+2$ or $|x|<1/\e^2-1$.
Moreover, in the inner integral we integrate only over the set $\{|y|>1\},$ since the integrand is zero on~$\{|y|<1\}$. 
Assume first that $|x|<1/\e^2-1$. Then $|x-y|<1/\e^2$ and by the definition of $a_{\e^2}$ we get that $a_{\e^2}(x-y)=0$ and the integrand is zero. 
In the case when $|x|>1/\e+2$ we obtain that $|x-y|>1/\e+1,$ hence~${g_\e(x)=g_\e(x-y)=g_\e'(x-y)=0}$ and the integrand is again zero.

It thus remains to estimate the latter integral.
Because of the relation $(1-a_\e)^{(k)}a_{\e^2}=0,$ $k=0,\,1$, we have that  $g_\e^{(k)}a_{\e^2}=0,$ $k=0,\,1$, and together with H\"older's inequality we get
\begin{align*}
\|\bA_2^\e[\vartheta]\|_2\leq\Big(\int_{\R}|g_\e(x)|^2 \Big(\int_{\{|y|>1\}}\frac{|(a_{\e^2}\vartheta)(x-y)|}{y^2} \, dy\Big)^{2}\, dx\Big)^{1/2}\leq C\|a_{\e^2}\vartheta\|_p\|g_\e\|_2
\leq C\|\vartheta\|_p\|g_\e\|_2.
\end{align*}
Due to the fact that $g_\e\in L_p(\R)$ with  $p\in (2,\infty)$ satisfies $g_\e=0$ in $\{|x|>1/\e+1\}$, we deduce that~$g_\e\in L_2(\R)$, and therewith we may conclude   $\bA_2^\e[\vartheta]  \in L_2(\R)$.
Summarizing, we have established the injectivity of $\lambda-\bA(f)^*$.\medskip

\noindent{Step 2.} We now prove there exists a constant $C_0>0$ such  that
 \begin{equation}\label{DE:EST12}
 \|(\lambda-\mathbb{A}(f)^*)[\vartheta]\|_p\geq C_0\|\vartheta\|_p \qquad \mbox{for all $\vartheta\in L_p(\mathbb{R})$ and $\lambda\in\R\setminus(-1,1)$}.
 \end{equation}
To this end, we argue by contradiction and assume that the claim \eqref{DE:EST12} is false. 
Then, there exists a   sequence~${(\vartheta_n)_n\subset L_p(\mathbb{R})}$ and a 
bounded sequence ${(\lambda_n)_n\subset\mathbb{R}}$ such that~${|\lambda_n|\geq1,}$ ${\|\vartheta_n\|_p=1}$ for all~${n\in\mathbb{N}}$, 
and with~${(\lambda_n-\mathbb{A}(f)^*)[\vartheta_n]=:\varphi_n\to0}$ in $L_p(\mathbb{R})$.
 After possibly extracting some subsequences,  we may assume that $\lambda_n\to\lambda$ in $\mathbb{R}$ and $\vartheta_n\rightharpoonup \vartheta$ in $L_p(\mathbb{R}).$
 Given $h\in L_{p'}(\mathbb{R})$, it holds, in view of $\bA(f)\in\kL(L_{p'}(\R))$, cf. \eqref{OpAB} and Lemma~\ref{L:MP0}~(i), that
 \[
 \langle(\lambda-\mathbb{A}(f)^*)[\vartheta] | h \rangle_{L_2}=
\langle\vartheta|(\lambda-\mathbb{A}(f))[h]\rangle_{L_2}
=\underset{n\to\infty}\lim\langle\vartheta_n|(\lambda_n-\mathbb{A}(f))[h]\rangle_{L_2}
=\underset{n\to\infty}\lim\langle \varphi_n|h\rangle_{L_2} =0.
\]
Hence, $(\lambda-\mathbb{A}(f)^*)[\vartheta]=0$, and the result established in Step 1 implies that~$\vartheta=0$. 

Since $|\lambda_n|\geq 1$ for all $n\in\mathbb{N}$, we may choose in virtue of Lemma \ref{L:comm'}  a constant $\e>0$ such 
that~${(\lambda_n-\chi_j^\varepsilon\mathbb{A}(f)^*\chi_j^\varepsilon),\, (\lambda-\chi_j^\varepsilon\mathbb{A}(f)^*\chi_j^\varepsilon)\in  {\rm Isom}(L_p(\mathbb{R}))}$  
 for all~$-N+1\leq j\leq N$ and all $n\in\mathbb{N}$.
Owing to
 \[
 1=\|\vartheta_n\|_p\leq \sum_{j=-N+1}^N\|\pi_j^\varepsilon\vartheta_n\|_p,
 \]
there exists an integer $j_*$ with ${-N+1\leq j_*\leq N}$ and a subsequence of $(\vartheta_n)_n$ (not relabeled)    such that 
 \begin{equation}\label{contr}
 \|\pi_{j_*}^\varepsilon\vartheta_n\|_p\geq (2N)^{-1} \qquad\mbox{for all $n\in\mathbb{N}$.}
 \end{equation}
Recalling the definition \eqref{Kjtheta} of the operators $K_{j}$,~$-N+1\leq j\leq N$, we deduce from the identity~${(\lambda_n-\mathbb{A}(f)^*)[\vartheta_n]=\varphi_n}$ the following  formula
 \begin{equation}\label{nice}
 \pi_{j_*}^\varepsilon \vartheta_n=(\lambda_n-\chi_{j_*}^\varepsilon\mathbb{A}(f)^*\chi_{j_*}^\varepsilon)^{-1} [K_{j_*}[ \vartheta_n]  ]
 +(\lambda_n-\chi_{j_*}^\varepsilon\mathbb{A}(f)^*\chi_{j_*}^\varepsilon)^{-1}[\pi_{j_*}^\varepsilon\varphi_n] 
\end{equation} 
for all $n\in\mathbb{N}$. 
Since $K_{j^*}$  is compact, cf. Lemma \ref{L:comp}, it follows that   $K_{j^*}[\vartheta_n]\to0$ in $L_p(\mathbb{R})$.
Also using the convergences $\lambda_n-\chi_{j_*}^\varepsilon\mathbb{A}(f)^*\chi_{j_*}^\varepsilon \to \lambda-\chi_{j_*}^\varepsilon\mathbb{A}(f)^*\chi_{j_*}^\varepsilon$ in $\mathcal{L}(L_p(\mathbb{R}))$ 
and~${\varphi_n\to0}$ in $L_p(\mathbb{R}),$ we now infer from \eqref{nice} that~$\pi_{j_*}^\varepsilon\vartheta_n\to 0$ in $L_p(\mathbb{R})$, which contradicts (\ref{contr}).
This proves \eqref{DE:EST12}.
 
 Since $\lambda-\mathbb{A}(f)^*\in{\rm Isom}(L_p(\mathbb{R}))$   for all $|\lambda|>\|\mathbb{A}(f)^*\|_{\kL(L_p(\mathbb{R}))}$, the estimate  \eqref{DE:EST12} and
   the   method of continuity, cf.~\cite[Proposition I.1.1.1]{Am95}, imply that 
 ${\lambda-\mathbb{A}(f)^*\in{\rm Isom}(L_p(\mathbb{R}))}$   for all~$\lambda\in\R\setminus(-1,1)$, and the proof is complete. 
\end{proof}

 Before proceeding, let us emphasize, as a straight forward consequence of Theorem~\ref{Thm:1}, that  the estimate \eqref{DE:EST12}  holds  for each $p\in (1,\infty)$ (possibly with a different constant~$C_0$).
 Using this estimate, we now provide the following invertibility result.

 \begin{prop}\label{Prop:2}
Given  $p\in(1,\infty)$, $s\in (1+1/p,2)$, and $f\in W^s_p(\R)$,   we have
\[
\lambda-\bA(f)^*\in{\rm Isom}(W^{s-1}_p(\R)) \qquad\text{for all $\lambda\in\R\setminus(-1,1).$}
\]
\end{prop}
\begin{proof} 
Given $\vartheta\in W^{s-1}_p(\mathbb{R})$ and $\lambda\in\R\setminus(-1,1)$, we set  $\varphi:=(\lambda-\mathbb{A}(f)^*)[\vartheta]$. 
The formula \eqref{adjBNM}, Lemma~\ref{L:MP0}~(ii), and the algebra property of $W^{s-1}_p(\R)$ entail that $\varphi\in W^{s-1}_p(\mathbb{R})$. 
Letting  $C_0>0$ denote the constant in \eqref{DE:EST12} and recalling \eqref{normwsp}, we have
\begin{equation}\label{frf}
\begin{aligned}
[\vartheta]^p_{W^{s-1}_p}
&=\int_{\mathbb{R}}\frac{\|\vartheta-\tau_{\xi}\vartheta\|_p^p}{|\xi|^{1+(s-1)p}}\, d{\xi}
\leq C_0^p\int_{\mathbb{R}}\frac{\|(\lambda-\mathbb{A}(f)^*)[\vartheta-\tau_{\xi} \vartheta ]\|_p^p}{|\xi|^{1+(s-1)p}}\, d{\xi}\\[1ex]
&\leq 2^pC_0^p\int_{\mathbb{R}}\frac{\|\varphi-\tau_{\xi}\varphi\|_p^p}{|\xi|^{1+(s-1)p}}\, d{\xi}
+2^pC_0^p\int_{\mathbb{R}}\frac{\|(\mathbb{A}(f)^* -\mathbb{A}(\tau_{\xi}f)^*)[\tau_\xi\vartheta])\|_p^p}{|\xi|^{1+(s-1)p}}\, d{\xi}\\[1ex]
&= 2^pC_0^p[(\lambda-\mathbb{A}(f)^*)[\vartheta]]^p_{W^{s-1}_p}+2^pC_0^p\int_{\mathbb{R}}\frac{\|(\mathbb{A}(f)^* -\mathbb{A}(\tau_{\xi}f)^*)[\tau_\xi\vartheta])\|_p^p}{|\xi|^{1+(s-1)p}}\, d{\xi}.
\end{aligned}
\end{equation}
In order to  estimate the integrand in last term of the latter inequality we infer from \eqref{adjBNM} that
\begin{align*}
\|\mathbb{A}(f)^* -\mathbb{A}(\tau_{\xi}f)^*)[\tau_\xi\vartheta]\|_p&\leq \|(B^0_{1,1}(f) -B_{1,1}^0(\tau_{\xi}f))[\tau_\xi\vartheta]\|_p\\[1ex]
&\hspace{0,45cm}+\|B_{0,1}(f)[f'\tau_\xi\vartheta]-B_{0,1}(\tau_{\xi}f)[\tau_{\xi}(f'\vartheta)]\|_p.
\end{align*}
Let now $s'\in(1+1/p,s)$ be chosen. 
The relation \eqref{rell}, Lemma \ref{L:MP0}~(i) and~(iii) (with $s=s'$),  and the embedding  $W^{s'-1}_p(\R)\hookrightarrow {\rm BUC}(\R)$ lead us to
 \begin{align*}
\|(B^0_{1,1}(f) -B_{1,1}^0(\tau_{\xi}f))[\tau_\xi\vartheta]\|_p
&\leq\|B_{1,1}(\tau_{\xi}f)[f-\tau_{\xi}f,\tau_\xi\vartheta]\|_p\\[1ex]
&\qquad+\|B_{3,2}(f,\tau_{\xi}f)[f+\tau_{\xi}f, f-\tau_{\xi}f,f,\tau_\xi\vartheta]\|_p\\[1ex]
&\leq C\|f'-\tau_{\xi}f'\|_p\|\vartheta\|_{W^{s'-1}_p}, 
\end{align*}
and
 \begin{align*}
&\hspace{-0.5cm}\|B_{0,1}(f)[f'\tau_\xi\vartheta]-B_{0,1}(\tau_{\xi}f)[\tau_{\xi}(f'\vartheta)]\|_p\\[1ex]
&\leq \|B_{0,1}(\tau_{\xi}f)[(f'-\tau_{\xi}f')\tau_\xi\vartheta)]\|_p++\|B_{2,2}(f,\tau_{\xi}f)[f+\tau_{\xi}f, f-\tau_{\xi}f,f'\tau_{\xi}\vartheta]\|_p\\[1ex]
&\leq C\|f'-\tau_{\xi}f'\|_{p}\|\vartheta\|_{W^{s'-1}_p} 
\end{align*}
for all $\xi\in\R$, where $C>0$ depends only on $f$.
These estimates together with \eqref{frf} and ~\eqref{DE:EST12} imply there exists a constant $C_1$ (which depends only on $f$) such that 
\begin{align}\label{L:ES}
\|\vartheta\|_{W^{s-1}_p}\leq C_1(\|(\lambda-\mathbb{A}(f)^*)[\vartheta]\|_{W^{s-1}_p}+\|\vartheta\|_{W^{s'-1}_p}),\qquad \vartheta\in W^{s-1}_p(\mathbb{R}).
\end{align} 
The  interpolation property \eqref{IP} and Young's inequality now imply there exists $C>0$ with
 \[\|\vartheta\|_{W^{s'-1}_p}\leq \frac{1}{2C_1}\|\vartheta\|_{W^{s-1}_p}+C\|\vartheta\|_p,\qquad \vartheta\in W^{s-1}_p(\mathbb{R}).\]
In view of this estimate and relying also on \eqref{L:ES} and  \eqref{DE:EST12}, we may find a constant $C$ such that 
\begin{align*}
\|\vartheta\|_{W^{s-1}_p}\leq C \|(\lambda-\mathbb{A}(f)^*)[\vartheta]\|_{W^{s-1}_p}, \quad\mbox{$\vartheta\in W^{s-1}_p(\mathbb{R})$ and $\lambda\in\R\setminus(-1,1),$}
\end{align*}
The method of continuity \cite[Proposition I.1.1.1]{Am95} leads  now, similarly as in the proof of Theorem~\ref{Thm:1}, to the desired conclusion.
 \end{proof}
 
As a final result of  this section we establish the following  corollary.
\begin{cor}\label{Cor:1}
Given  $f\in W^2_p(\R)$, we have $\lambda-\bA(f)^* \in{\rm Isom}(W^{1}_p(\R)) $ for all $\lambda\in\R\setminus(-1,1)$.
\end{cor}
\begin{proof}
Given $\vartheta\in W^1_p(\R)$, we infer from Proposition~\ref{Prop:1} that  $\bA(f)^*[\vartheta]$ belongs to $ W^1_p(\R)$  and its derivative satisfies $(\bA(f)^*[\vartheta])'=-\bA(f)[\vartheta'].$
This property,  \cite[Theorem 3 and Theorem 4]{AM22}, and   \eqref{DE:EST12} imply there exists a constant $C>0$ such that 
\begin{align*}
\|(\lambda-\bA(f)^*)[\vartheta]\|_{W^1_p}^p&=\|(\lambda-\bA(f)^*)[\vartheta]\|_{p}^p+\|((\lambda-\bA(f)^*)[\vartheta])'\|_{p}^p\\[1ex]
&=\|(\lambda-\bA(f)^*)[\vartheta]\|_{p}^p+\|(\lambda+\bA(f))[\vartheta']\|_{p}^p\\[1ex]
&\geq C\|\vartheta\|_{W^1_p}^p
\end{align*}
for all $\lambda\in\R\setminus(-1,1)$ and all  $\vartheta\in W^1_p(\R)$.
The claim follows now from this estimate via the method of continuity.
\end{proof}

\section{The quasilinear evolution problem and the  proof of the main result}\label {Sec:4}

As a first step we take advantage of the quasilinear character of the curvature operator to reformulate, in virtue of Proposition~\ref{Prop:2}, 
the evolution problem \eqref{P:2} as a quasilinear evolution problem for~$f$ in a suitable functional analytic setting.
More precisely, we recast \eqref{P:2} as the evolution problem
\begin{equation}\label{NNEP}
\frac{df}{dt}(t)=\Phi(f(t))[f(t)],\quad t>0,\qquad f(0)=f_0,
\end{equation}
where the nonlinear operator $[f\mapsto\Phi(f)]:W^s_p(\R)\to \kL(W^{s+1}_p(\R),W^{s-2}_p(\R))$, with $p\in(1,\infty)$ and~${s\in (1+1/p,2)}$, is defined as follows.
Given $f\in W^s_p(\R)$ and $h\in W^{s+1}_p(\R)$, let 
\begin{align}\label{qlk}
\kappa(f)[h]:=\frac{h''}{(1+f'^2)^{3/2}}.
\end{align}
If $f\in W^2_p(\R)$, then  $\kappa(f)[f]$ coincides with  the curvature $\kappa(f)$ of the interface $\{y=f(x)\}$.
This operator is smooth, that is
 \begin{equation}\label{Reg:kappa}
 \kappa\in {\rm C}^\infty(W^s_p(\R),\kL(W^{s+1}_p(\R), W^{s-1}_p(\R))),
 \end{equation}
see, e.g. \cite[Lemma 3.1]{MM21}.
In virtue of Proposition~\ref{Prop:2} and of $a_\mu\in(-1,1)$, there exits, for given functions~$f\in W^s_p(\R)$ and $h\in W^{s+1}_p(\R)$,  a unique solution~${\vartheta:=\vartheta(f)[h]\in W^{s-1}_p(\R)}$ to the equation
\begin{equation}\label{theta}
(1+a_\mu\bA(f)^*)[\vartheta]=b_\mu \big(\sigma\kappa(f)[h]-\Theta h\big).
\end{equation}
A direct consequence of \eqref{Reg:kappa}, of the smoothness of the map which associate to an isomorphism its inverse, of the representation \eqref{adjBNM} of $\bA(f)^*$, and of the property
\begin{equation}\label{Reg:bnm}
[f\mapsto B^0_{n,m}(f)]\in {\rm C}^\infty(W^s_p(\R),\kL(W^{s-1}_p(\R))),
\end{equation}
we obtain that 
\begin{equation}\label{Reg:theta}
 \vartheta\in {\rm C}^\infty(W^s_p(\R),\kL(W^{s+1}_p(\R), W^{s-1}_p(\R))).
 \end{equation}
The property \eqref{Reg:bnm} follows from Lemma~\ref{L:MP0}~(ii), by arguing as in   \cite[Appendix C]{MP2021} (where the case $p=2$ is proven in detail).
Having introduced these quasilinear operators, we now define the operator $\Phi$ by setting
\begin{equation}\label{Phi}
\Phi(f)[h]:=-(\bB(f)^*[\vartheta(f)[h]])'.
\end{equation}
In view of the continuity of the operator $d/dx: W^{s-1}_p(\R)\to W^{s-2}_p(\R)$, we deduce from \eqref{adjBNM}, \eqref{Reg:bnm}, and~\eqref{Phi} that 
\begin{equation}\label{Reg:Phi}
 \Phi\in {\rm C}^\infty(W^s_p(\R),\kL(W^{s+1}_p(\R), W^{s-2}_p(\R))).
 \end{equation}

As a second step, we prove in Theorem~\ref{Thm:2} below that the evolution problem \eqref{NNEP} is of parabolic type,  
that is we show that $\Phi(f)$ is, when viewed as an unbounded operator in $W^{s-2}_p(\R)$ with definition domain $W^{s+1}_p(\R)$, the generator of an analytic semigroup in $\kL(W^{s-2}_p(\R))$.
This property enables us to use the abstract quasilinear parabolic theory presented in~\cite{Am93, MW20} when establishing Theorem~\ref{MT1}.

\begin{thm}\label{Thm:2}
Given $p\in (1,\infty)$, $s\in(1+1/p,2)$, and $f\in W^{s}_p(\R)$, we have
 $$-\Phi(f)\in\kH(W^{s+1}_p(\R),W^{s-2}_p(\R)).$$
\end{thm}
In the proof of Theorem~\ref{Thm:2} we are inspired by the strategy used in the references  \cite{E94,  ES95,  ES97}.
The main step in the proof of Theorem~\ref{Thm:2} is provided by the localization result established in Proposition~\ref{Prop:3} below.
Before proceeding with this result, we  note that the leading order part~$\Phi^\pi(0)$ of~$\Phi(0)$ is the linear operator
\begin{equation}\label{Phi0}
\Phi^\pi(0)= \sigma b_\mu\frac{d}{dx}H\frac{d^2}{dx^2}=\sigma b_\mu H\frac{d^3}{dx^3},
\end{equation}
where $H$ denotes the Hilbert transform. 
We point out that    $\Phi^\pi(0)$ is the Fourier multiplier with symbol $[\xi\mapsto -\sigma b_\mu |\xi|^3]$.

\begin{prop}\label{Prop:3} 
Let $p\in(1,\infty)$, $1+1/p<s'<s<2$, $f\in W^{s}_p(\R)$, and~$\nu>0$ be given. 
Then, there exist $\e\in(0,1)$, a $\e$-locali\-za\-tion family  $\{(\pi_j^\e,x_j^\e)\,:\, -N+1\leq j\leq N\} $, and  a constant~${K=K(\e)}$, 
 such that 
 \begin{equation}\label{D1}
  \|\pi_j^\e \Phi(\tau f) [h]-\alpha_{\tau_,j}\Phi^\pi(0)[\pi_j^\e h]\|_{W^{s-2}_p}\leq \nu \|\pi_j^\e h\|_{W^{s+1}_p}+K\|  h\|_{W^{ s'+1}_p}
 \end{equation}
 for all $ -N+1\leq j\leq N$, $\tau\in[0,1],$  and  $h\in W^{s+1}_p(\R)$,  where, letting $\alpha_\tau:=(1+\tau^2f'^2)^{-3/2}$, we defined
\[ 
\alpha_{\tau,N}:=\lim_{|x|\to\infty}\alpha_{\tau}(x)=1\qquad\text{and}\qquad \alpha_{\tau,j}:= \alpha_\tau(x_j^\e),\quad |j|\leq N-1.
\]
\end{prop}
\begin{proof} 
In the following    $C$ and $C_i$, $0\leq i\leq 2$, denote constants that do not depend on $\e$, while constants denoted by $K$ may depend on $\e$.

Setting $C_0:=\|d/dx\|_{\kL( W^{s-1}_p(\R), W^{s-2}_p(\R))} $, for given $-N+1\leq j\leq N$, $\tau\in[0,1],$  and~$h\in W^{s+1}_p(\R)$ we obtain,
 in view of \eqref{Phi} and of the representation~\eqref{Phi0} of~$\Phi^\pi(0)$, that 
\begin{equation}\label{ES0}
\begin{aligned}
\|\pi_j^\e \Phi(\tau f) [h]-\alpha_{\tau_,j}\Phi^\pi(0)[\pi_j^\e h]\|_{W^{s-2}_p}&=\|\pi_j^\e \Phi(\tau f) [h]-\sigma b_\mu\alpha_{\tau_,j}(H[(\pi_j^\e h)''])'\|_{W^{s-2}_p}\\[1ex]
&\leq \|(\pi_j^\e \bB(\tau f)^*[\vartheta(\tau f)[h]]+\sigma b_\mu\alpha_{\tau_,j}H[(\pi_j^\e h)''])'\|_{W^{s-2}_p}\\[1ex]
&\qquad+\|(\pi_j^\e)'\bB(\tau f)^*[\vartheta(\tau f)[h]]\|_{W^{s-2}_p}\\[1ex]
&\leq C_0 \|\pi_j^\e \bB(\tau f)^*[\vartheta(\tau f)[h]]+\sigma b_\mu\alpha_{\tau_,j}H[(\pi_j^\e h)'']\|_{W^{s-1}_p}\\[1ex]
& \qquad+ K\|  h\|_{W^{ s'+1}_p},
\end{aligned}
\end{equation}
since, by Lemma~\ref{L:MP0}~(i) and Theorem~\ref{Thm:1}, we have
\begin{align*}
\|(\pi_j^\e)'\bB(\tau f)^*[\vartheta(\tau f)[h]]\|_{W^{s-2}_p}&\leq K\| \bB(\tau f)^*[\vartheta(\tau f)[h]]\|_{p}\leq K\|\vartheta(\tau f)[h]\|_p\leq K\|h\|_{W^2_p}\leq K\|  h\|_{W^{ s'+1}_p}.
\end{align*}
It thus remains to estimate the expression $ \|\pi_j^\e \bB(\tau f)^*[\vartheta(\tau f)[h]]+\sigma b_\mu\alpha_{\tau_,j}H[(\pi_j^\e h)'']\|_{W^{s-1}_p}.$
 This is done in several steps.\medskip
 
  \noindent{Step 1.} We  prove there exists a positive constant $C_1$ such that 
 \begin{equation}\label{x-1}
 \|\pi_j^\e\vartheta(\tau f)[h]\|_{W^{s-1}_p}\leq C_1\|\pi_j^\e h\|_{W^{s+1}_p}+ K\|h\|_{W^{s'+1}_p}
 \end{equation}
 for all $\e\in(0,1)$, $\tau\in[0,1]$, $-N+1\leq j\leq N$, and $h\in  W^{s-1}_p(\R)$.
 To this end we multiply \eqref{theta} by~$\pi_j^\e$ to deduce that
 \begin{align*} 
(1+a_\mu\bA(\tau f)^*)[\pi_j^\e \vartheta(\tau f)[h]]&=b_\mu \pi_j^\e\big(\sigma\kappa(\tau f)[h]-\Theta h\big)\\[1ex]
&\quad+a_\mu\big(\bA(\tau f)^*[\pi_j^\e \vartheta(\tau f)[h]]-\pi_j^\e\bA(\tau f)^*[ \vartheta(\tau f)[h]]\big).
\end{align*}
The first term on the right is estimated as follows 
\[
\|\pi_j^\e\big(\sigma\kappa(\tau f)[h]-\Theta h\big)\|_{W^{s-1}_p}\leq C\|\pi_j^\e h\|_{W^{s+1}_p}+ K\|h\|_{W^{s'+1}_p}.
\]
Moreover, combining \eqref{adjBNM} and  the commutator estimate in Lemma~\ref{L:AL1},   the arguments used to derive~\eqref{ES0} yield 
\[
\|\bA(\tau f)^*[\pi_j^\e \vartheta(\tau f)[h]]-\pi_j^\e\bA(\tau f)^*[ \vartheta(\tau f)[h]]\|_{W^{s-1}_p}\leq K\|\vartheta(\tau f)[h]\|_p \leq  K\|  h\|_{W^{ s'+1}_p}.
\]
 Proposition~\ref{Prop:2} now ensures that \eqref{x-1} indeed holds true. 
 \medskip

 \noindent{Step 2.}  
 Taking advantage of Lemma~\ref{L:AL2} (if $|j|\leq N-1$) and Lemma~\ref{L:AL3} (if~$ j=N$), we infer from~\eqref{adjBNM} that for each sufficiently small $\e\in(0,1)$ we have  
 \begin{equation*}
 \|\pi_j^\e \bB(\tau f)^*[\vartheta(\tau f)[h]]+H[\pi^\e_j  \vartheta(\tau f)[h]]\|_{W^{s-1}_p}\leq \frac{\nu}{2C_0C_1}\|\pi^\e_j\vartheta(\tau f)[h]\|_{W^{s-1}_p}+K\|\vartheta(\tau f)[h]\|_{_{W^{s'-1}_p}}
\end{equation*}  
for all $\tau\in[0,1]$, $-N+1\leq j\leq N$, and $h\in  W^{s+1}_p(\R)$.
Using \eqref{x-1} and \eqref{Reg:theta} (with $s=s'$), we infer from the latter estimate that
 \begin{equation}\label{x-2}
 \|\pi_j^\e \bB(\tau f)^*[\vartheta(\tau f)[h]]+H[\pi^\e_j  \vartheta(\tau f)[h]]\|_{W^{s-1}_p}\leq \frac{\nu}{2 C_0}\|\pi^\e_j h\|_{W^{s+1}_p}+K\|h\|_{_{W^{s'+1}_p}}
\end{equation}  
for all $\tau\in[0,1]$, $-N+1\leq j\leq N$, and $h\in  W^{s+1}_p(\R)$, provided that $\e$ is  sufficiently small.\medskip

 \noindent{Step 3.} We first note that  $\alpha_\tau(f)=(1+\tau^2 f'^2)^{-3/2}\in{\rm BUC}^{s-1-1/p}(\R)$, $\tau\in[0,1]$, and~${{\alpha_\tau(f)(x)\to1}}$ for~$|x|\to\infty$ (uniformly with respect to $\tau\in[0,1]$).
 In this step we show that, for each sufficiently small~${\e\in(0,1),}$ we have 
 \begin{equation}\label{x-3}
\|H[\pi^\e_j  \vartheta(\tau f)[h]]-\sigma b_\mu\alpha_{\tau_,j}H[(\pi_j^\e h)'']\|_{W^{s-1}_p}\leq \frac{\nu}{2 C_0}\|\pi^\e_j h\|_{W^{s+1}_p}+K\|h\|_{_{W^{s'+1}_p}} 
\end{equation}  
for all $\tau\in[0,1]$, $-N+1\leq j\leq N$, and  $h\in  W^{s+1}_p(\R)$.

To proceed, we set   $C_2:=\|H\|_{\kL(W^{s-1}_p(\R))}$ and  obtain that
\[ 
 \|H[\pi^\e_j  \vartheta(\tau f)[h]]-\sigma b_\mu\alpha_{\tau_,j}H[(\pi_j^\e h)'']\|_{W^{s-1}_p}\leq C_2\|\pi^\e_j  \vartheta(\tau f)[h]-\sigma b_\mu\alpha_{\tau_,j}(\pi_j^\e h)''\|_{W^{s-1}_p}. 
 \]
In order to estimate the right side of the latter inequality, we infer from \eqref{theta} that 
\begin{equation*}
\pi^\e_j\vartheta(\tau f)[h]-\sigma b_\mu\alpha_{\tau_,j}(\pi_j^\e h)''=-a_\mu\pi^\e_j\bA(\tau f)^*[\vartheta(\tau f)[h]]
+b_\mu \big[\pi^\e_j\big(\sigma\kappa(\tau f)[h]-\Theta h\big)-\sigma \alpha_{\tau_,j}(\pi_j^\e h)''\big].
\end{equation*} 
 Using Lemma~\ref{L:AL2} (if $|j|\leq N-1$) and Lemma~\ref{L:AL3} (if $ j=N$) together with the representation \eqref{adjBNM} of $\bA(\tau f)^*$, we have 
 \[
 \|\pi^\e_j\bA(\tau f)^*[\vartheta(\tau f)[h]]\|_{W^{s-1}_p}\leq \frac{\nu}{4(1+|a_\mu|)  C_0C_1C_2}\|\pi^\e_j\vartheta(\tau f)[h]\|_{W^{s-1}_p}+K\|\vartheta(\tau f)[h]\|_{_{W^{s'-1}_p}}
 \]
 and, arguing as in the derivation of \eqref{x-2}, we conclude that 
 \[
 \|\pi^\e_j\bA(\tau f)^*[\vartheta(\tau f)[h]]\|_{W^{s-1}_p}\leq \frac{\nu}{4(1+|a_\mu|) C_0 C_2}\|\pi^\e_j h\|_{W^{s+1}_p}+K\|h\|_{_{W^{s'+1}_p}}
 \]
 for all $\tau\in[0,1]$, $-N+1\leq j\leq N$, and $h\in  W^{s+1}_p(\R)$, provided that $\e$ is sufficiently small.
 Moreover, for all sufficiently small $\e\in(0,1),$ we obtain in view of  $\chi_j^\e\pi_j^\e=\pi_j^\e$, of the estimate \eqref{MES}, and of the H\"older continuity of $\alpha_\tau$ that 
 \begin{align*}
\|\pi^\e_j\big(\sigma\kappa(\tau f)[h]-\Theta h\big)-\sigma  \alpha_{\tau_,j} (\pi_j^\e h)'' \|_{W^{s-1}_p}&\leq \sigma\|(\alpha_\tau-\alpha_\tau(x_j^\e))(\pi_j^\e h)''\|_{W^{s-1}_p}+K\|h\|_{W^{s'+1}_p} \\[1ex]
 &\leq 2\sigma\|\chi_j^\e(\alpha_\tau-\alpha_\tau(x_j^\e))\|_\infty\|\pi_j^\e h\|_{W^{s+1}_p}+K\|h\|_{W^{s'+1}_p} \\[1ex]
 &\leq \frac{\nu}{4b_\mu C_0 C_2} \|\pi_j^\e h\|_{W^{s+1}_p}+K\|h\|_{W^{s'+1}_p} 
 \end{align*}
  for all $\tau\in[0,1]$, $|j|\leq N-1$, and $h\in  W^{s+1}_p(\R)$.
Since $\alpha_\tau-1$ vanishes at infinity, for $j=N$ we similarly have
  \begin{align*}
\|\pi^\e_N\big(\sigma\kappa(\tau f)[h]-\Theta h\big)-\sigma  \alpha_{\tau_,N}(\pi_N^\e h)''\|_{W^{s-1}_p} &\leq \sigma\|(\alpha_\tau-1)(\pi_N^\e h)''\|_{W^{s-1}_p}+K\|h\|_{W^{s'+1}_p} \\[1ex]
 &\leq 2\sigma\|\chi_N^\e(\alpha_\tau-1)\|_\infty\|\pi_N^\e h\|_{W^{s+1}_p}+K\|h\|_{W^{s'+1}_p} \\[1ex]
 &\leq \frac{\nu}{4b_\mu C_0 C_2} \|\pi_N^\e h\|_{W^{s+1}_p}+K\|h\|_{W^{s'+1}_p}. 
 \end{align*}
 Gathering all these estimates, we conclude that   \eqref{x-3} is valid.\medskip
 Combining  the estimates \eqref{ES0}, \eqref{x-2}, and \eqref{x-3}, we arrive at \eqref{D1} and the proof is complete.
\end{proof}

We are now in a position to establish Theorem~\ref{Thm:2}.
The proof relies deeply on Proposition~\ref{Prop:3}, the arguments being  identical to those in the proof \cite[Theorem~3.5]{MM21}.

\begin{proof}[Proof of Theorem~\ref{Thm:2}] 
Since $f'$ is a bounded  function, there exists a constant~$\kappa_0\geq 1$ such that the Fourier multipliers $\alpha_{\tau_,j}\Phi^\pi(0)$, $\tau\in[0,1]$ and $-N+1\leq j\leq N$, identified in Proposition~\ref{Prop:2} satisfy
 \begin{align}
\bullet &\,\, \text{$\lambda-\alpha_{\tau_,j}\Phi^\pi(0)\in {\rm Isom}(W^{s+1}_p(\R),W^{s-2}_p(\R))$   for all  $\re\lambda\geq 1$},\label{L:FM1}\\[1ex]
\bullet &\,\,  \kappa_0\|(\lambda-\alpha_{\tau_,j}\Phi^\pi(0))[h]\|_{W^{s-2}_p}\geq |\lambda|\cdot\|h\|_{W^{s-2}_p}+\|h\|_{W^{s+1}_p} \quad  \text{for all  $h\in W^{s+1}_p(\R)$,  $\re\lambda\geq 1$.}\label{L:FM2}
\end{align}
The properties \eqref{L:FM1}-\eqref{L:FM2} together with the estimate~\eqref{D1} in Proposition~\ref{Prop:3} lead now to the desired result, see \cite[Theorem~3.5]{MM21} for full details.
\end{proof}

We conclude this section with the proof of our main result which uses to a large extent the abstract quasilinear parabolic theory presented in \cite{Am93} (see also \cite[Theorem 1.1]{MW20}).
\begin{proof}[Proof of Theorem~\ref{MT1}]
To start, let 
\[
0< \beta:=\frac{2}{3}<\alpha:=\frac{s-\ov s+2}{3}<1. 
\]
We further define $E_1:=W^{\ov s+1}_p(\R)$, $E_0:=W^{\ov s-2}_p(\R)$, and $E_\eta:=(E_0,E_1)_{\eta,p} $ for $\eta\in(0,1)$.
Recalling the interpolation property \eqref{IP}, we have $E_\alpha=W^{s}_p(\R)$ and $E_\beta=W^{\ov s}_p(\R)$.
We now infer from~\eqref{Reg:Phi} and Theorem~\ref{Thm:2} (both with~${s=\ov s}$), that $-\Phi \in {\rm C}^{\infty}(E_\beta,\kH(E_1, E_0))$.
 Hence, the assumptions of \cite[Theorem 1.1]{MW20} are satisfied in the context of the Muskat problem~\eqref{NNEP}.
Consequently, given~${f_0\in W^s_p(\R)}$, there exists a unique maximal classical solution  $f= f(\,\cdot\, ; f_0)$ to \eqref{NNEP} such that
 \begin{equation*}  
 f\in {\rm C}([0,T^+),W^s_p(\mathbb{R}))\cap {\rm C}((0,T^+), W^{{\ov s}+1}_p(\mathbb{R}))\cap {\rm C}^1((0,T^+), W^{{\ov s}-2}_p(\mathbb{R})) 
  \end{equation*}
   and  
  \[
f\in   {\rm C}^{\zeta}([0,T^+), W^{\ov s}_p(\mathbb{R}),
  \]
  where $T^+=T^+(f_0)\in(0,\infty]$ is the maximal existence time and the H\"older exponent $\zeta\in(0,\alpha-\beta]$  can be chosen arbitrary small, cf. \cite[Remark 1.2~(ii)]{MW20}.
  Moreover, the mapping $[(t,f_0)\mapsto f(t;f_0)]$ defines a semiflow on $W^s_p(\R)$ which is smooth in the open set
  \[
\{(t,f_0)\,:\, f_0\in W^s_p(\R),\, 0<t<T^+(f_0)\}\subset \R\times W^s_p(\R).  
  \]
  
  In remains to establish the parabolic smoothing property~\eqref{eq:fg}.
 To this end we note that  we may  chose above $E_0:= L_p(\mathbb{R})$ and  $E_1:=W^{3}_p(\mathbb{R})$.
 Similarly as in the case~${p=2}$ considered in \cite{MBV18},  using Lemma~\ref{L:MP1} and Theorem~\ref{Thm:1}, we may deduce that~${-\Phi\in {\rm C}^\infty(W^2_p(\R), \kH(W^3_p(\R),L_p(\R))).}$
Applying the quasilinear parabolic theory from~\cite{Am93} in this context and  arguing as  in the proof of \cite[Theorem~1.1]{MBV18},  we obtain that, given~${f_0\in W^{\ov s+1}_p(\R)}$, 
there exists a unique solution $\wt f=\wt f(\cdot; f_0)$ to \eqref{NNEP} that satisfies
\begin{equation}\label{eq:gf}
 \wt f\in {\rm C}([0,\widetilde T^+),W^{ \ov s+1}_p(\mathbb{R}))\cap {\rm C}((0,\widetilde T^+), W^{3}_p(\mathbb{R}))\cap {\rm C}^1((0,\widetilde T^+),L_p(\mathbb{R})),
  \end{equation}
  with $\wt T^+=\wt T^+(f_0)\in(0,\infty]$ denoting the maximal existence time.
  Additionally,~${[(t,f_0)\mapsto \wt f(t;f_0)]}$ defines a semiflow on~${W^{\ov s+1}_p(\R)}$ and 
  \begin{equation*} 
 \wt f\in {\rm C}^\infty((0,\wt T^+ )\times\R,\R)\cap {\rm C}^\infty ((0,\wt T^+),  W^k_p(\R))\quad \text{for all $k\in\N$}.
  \end{equation*}
 Recalling \eqref{eq:gf}, Lemma~\ref{L:MP0}~(i) and Theorem~\ref{Thm:1}  imply that $d\wt f/dt\in {\rm C}([0,\wt T^+),W_p^{-1}(\R)).$
 This property, \eqref{eq:gf}, and  the mean value theorem  lead now to $\wt f\in   {\rm C}^{\zeta}([0,\wt T^+), W^{\ov s}_p(\mathbb{R}))$ for some sufficiently small~$\zeta \in (0,\alpha-\beta]$.
 Consequently,~$\wt T^+\leq T^+ $ and  $\wt f(\cdot; f_0)=  f(\cdot; f_0)$ on $[0, \wt T^+)$ for all~${f_0\in W^{\ov s+1}_p(\R)}$.
 
 It actually holds $\wt T^+=T^+. $ Indeed, let us assume  that $\wt T^+<T^+$. 
 The properties established above enable us to conclude that  $\wt f\in   {\rm C}^{\zeta'}([0,\wt T^+), W^{2}_p(\mathbb{R}))$ for some sufficiently small~$\zeta' \in (0,1)$.
 Since $f\in {\rm C}((0,T^+), W^{{\ov s}+1}_p(\mathbb{R}))$, we may infer from \cite[Proposition~2.1]{MW20} (after possibly choosing a smaller $\xi'$),
 that there exist positive constants $\varepsilon>0$ and $\delta>0$ such that 
 for all $|t_0-\wt T^+|\leq \varepsilon,$ the problem
 \[
\frac{df}{dt}(t)=\Phi(f(t))[f(t)],\quad t>0,\qquad f(0)=f(t_0),
 \]
 has a unique solution
 \begin{equation*}
 \widehat f\in {\rm C}([0,\delta],W^{\ov s+1}_p(\mathbb{R}))\cap {\rm C}((0,\delta], W^{3}_p(\mathbb{R}))\cap {\rm C}^1((0,\delta],L_p(\mathbb{R}))\cap{\rm C}^{\zeta'}([0,\delta], W^{2}_p(\mathbb{R})).
  \end{equation*} 
  Hence, choosing $t_0<\wt T^+$ such that $t_0+\delta>\wt T^+$, we may extend $\wt f$ to the interval $[0,t_0+\delta)$, which contradicts the maximality property of $\wt f$, hence $\wt T^+=T^+$.
 This provides the regularity property~\eqref{eq:fg}  and the proof is complete.
\end{proof}

\appendix
\section{Localization of the singular integral operators $B_{n,m}^0(f)$}\label{Appendix A}

In this section we collect some results that enable us to localize the singular integrals operators~$B_{n,m}^0(f)$ introduced in~\eqref{defB0}.  
Lemma~\ref{L:AL2} and Lemma~\ref{L:AL3} can be viewed as generalizations of the method of freezing the coefficients of elliptic differential operators, and, together with the commutator estimate from Lemma~\ref{L:AL1},
 are essential in the proof of Proposition~\ref{Prop:3}.
 Before proceeding, we recall that $H$ denotes the Hilbert transform.
 
\begin{lemma}\label{L:AL1} 
Let $n,\, m \in \N$,  $p\in(1, \infty)$, $s\in(1+1/p,2)$, $f\in W^s_p(\R)$, and  ${\varphi\in {\rm BUC}^1(\R)}$ be given. 
Then, there exists  a constant $K$ that  depends only on~$ n,$~$m, $~$\|\varphi'\|_\infty, $ and~$\|f\|_{W^s_p}$  such that 
 \begin{equation*} 
  \|\varphi B_{n,m}^0(f)[\vartheta]- B_{n,m}^0(f)[ \varphi \vartheta]\|_{W^1_p}\leq K\| \vartheta\|_{p}
 \end{equation*}
for all   $\vartheta\in L_p(\R)$.
\end{lemma}
\begin{proof}
See \cite[Lemma 12]{AM22}.
\end{proof}

The next lemmas describe how to localize the operators $B_{n,m}^0(f)$.
\begin{lemma}\label{L:AL2} 
Let $n,\, m \in \N$, $1+1/p<s'<s<2$, and  $\nu\in(0,\infty)$ be given. 
Let further~${f\in W^s_p(\mathbb{R})}$ and   $a \in \{1\}\cup W^{s-1}_p(\mathbb{R})$.
For any sufficiently small $\e\in(0,1)$, there exists
a constant $K$ that depends only on $\e,\, n,\, m,\, \|f\|_{W^s_p},$  and  $\|a\|_{W^{s-1}_p}$ (if $a\neq1$)   such that 
 \begin{equation*} 
  \Big\|\pi_j^\e  B_{n,m}^0(f)[ a\vartheta]-\frac{a(x_j^\e) (f'(x_j^\e))^n}{[1+(f'(x_j^\e))^2]^m}H[\pi_j^\e \vartheta]\Big\|_{W^{s-1}_p}\leq \nu \|\pi_j^\e  \vartheta\|_{W^{s-1}_p}+K\| \vartheta\|_{W^{s'-1}_p} 
 \end{equation*}
for all $|j|\leq N-1$ and  $\vartheta\in W^{s-1}_p(\mathbb{R})$.
\end{lemma}  
\begin{proof}
If $a=1$, see \cite[Lemma~13]{AM22}. If $a\in W^{s-1}_p(\R)$, this result follows by arguing as in the proof of \cite[Lemma D.5]{MP2021} (where the result in the case $p=2$ is established).
\end{proof}

 Lemma~\ref{L:AL3}  below describes how to localize the operators $B^0_{n,m}(f)$ at infinity.

\begin{lemma}\label{L:AL3} 
Let $n,\, m \in \N$,  $1+1/p<s'<s<2$, and  $\nu\in(0,\infty)$ be given. 
Let further~${f\in W^s_p(\mathbb{R})}$    and~${a\in\{1\}\cup  W^{s-1}_p(\mathbb{R})}$.
For any sufficiently small  $\e\in(0,1)$, there exists a constant~$K$ that depends only on~$\e,\, n,\, m,\, \|f\|_{W^s_p},$   and $\|a\|_{W^{s-1}_p}$ (if $a\neq1$)  such that 
  \begin{equation*}
  \|\pi_N^\e  B_{n,m}^0(f)[a \vartheta]\|_{W^{s-1}_p}\leq \nu \|\pi_N^\e \vartheta\|_{W^{s-1}_p}+K\| \vartheta\|_{W^{s'-1}_p}\qquad \text{if $n\geq 1$ or $a\in W^{s-1}_p(\R)$},
 \end{equation*} 
  and
  \begin{equation*}
  \|\pi_N^\e  B_{0,m}^0(f)[\vartheta]-H[\pi_N^\e \vartheta]\|_{W^{s-1}_p}\leq \nu \|\pi_N^\e \vartheta\|_{W^{s-1}_p}+K\| \vartheta\|_{W^{s'-1}_p}
 \end{equation*} 
 for all $\vartheta\in W^{s-1}_p(\mathbb{R})$.
\end{lemma}  
\begin{proof}
If $a=1$, see \cite[Lemma~15]{AM22}. If $a\in W^{s-1}_p(\R)$, the desired estimate follows by arguing as in the proof of \cite[Lemma D.6]{MP2021} (where the result in the case $p=2$ is established).
\end{proof}

\bibliographystyle{siam}
\bibliography{B}

\end{document}